\documentclass[11pt,a4paper]{amsart}

\usepackage{amsfonts,amsmath,amssymb,amsthm}

\usepackage[utf8]{inputenc}
\usepackage[T1]{fontenc}

\newcommand{\Nn}{\mathbb{N}}
\newcommand{\Zz}{\mathbb{Z}}
\newcommand{\Qq}{\mathbb{Q}}
\newcommand{\Ff}{\mathbb{F}}
\newcommand{\Aa}{\mathbb{A}}

\renewcommand {\leq}{\leqslant}
\renewcommand {\geq}{\geqslant}
\renewcommand {\le}{\leqslant}
\renewcommand {\ge}{\geqslant}

\renewcommand{\ast}{\star}

{\theoremstyle{plain}
\newtheorem{theorem}{Theorem}[section]    
\newtheorem{lemma}[theorem]{Lemma}       
\newtheorem{proposition}[theorem]{Proposition}      
\newtheorem*{conjecture*}{Conjecture}
\newtheorem*{theorem*}{Theorem}
}
{\theoremstyle{remark}

\newtheorem{definition}[theorem]{Definition}      
\newtheorem*{remark*}{Remark}  
\newtheorem{remark}[theorem]{Remark}   
\newtheorem{example}[theorem]{Example}

}

\usepackage[a4paper]{geometry}
\def\CopSch{${\mathcal CopSch}(1,1)$}

\title[The Hilbert-Schinzel specialization property]{The Hilbert-Schinzel specialization property}

\author{Arnaud Bodin}
\author{Pierre D\`ebes}
\author{Joachim K\"onig}
\author{Salah Najib}

\email{arnaud.bodin@univ-lille.fr}
\email{pierre.debes@univ-lille.fr}
\email{jkoenig@knue.ac.kr}
\email{slhnajib@gmail.com}

\address{Universit\'e de Lille, CNRS, UMR 8524, Laboratoire Paul Painlev\'e, F-59000 Lille, France}
\address{Universit\'e de Lille, CNRS, UMR 8524, Laboratoire Paul Painlev\'e, F-59000 Lille, France}
\address{Department of Mathematics Education, Korea National University of Education, 28173 Cheongju, South Korea}
\address{Equipe ETRES, Facult\'e Polydisciplinaire de Khouribga, Universit\'e Sultan Moulay Slimane, BP 145, Hay Ezzaytoune, 25000 Khouribga, Maroc.}



\thanks{{\it Acknowledgment}. 
The first and second authors were supported by the Labex CEMPI  (ANR-11-LABX-0007-01).
The first author was also supported by the ANR project ``LISA'' (ANR-17-CE40--0023-01), and
thanks the University of British Columbia for his visit at PIMS}

\date{\today}

\begin{document}

\begin{abstract} 
We establish a version ``over the ring'' of the celebrated Hilbert Irreducibility Theorem. Given finitely many polynomials in $k+n$ variables, 
with coefficients in $\Zz$, of positive degree in the last $n$ variables, we show that if they are irreducible over $\Zz$ and satisfy a necessary  ``Schinzel condition'', then the first $k$ variables can be specialized in a Zariski-dense subset of $\Zz^k$ in such a way that irreducibility over $\Zz$ is preserved 
for the polynomials in the remaining $n$ vari\-ables. The Schinzel condition, which comes from the Schinzel Hypothesis, is that, when specializing the first $k$ variables in $\Zz^k$, the product of the polynomials should not always be divisible by some common prime number. Our result also improves on a ``coprime'' version of the Schinzel Hypothesis: under some Schinzel condition, coprime polynomials 
assume coprime values.
We prove our results over many other rings than $\Zz$, \hbox{e.g.} \hbox{UFDs and Dedekind domains.}
\end{abstract}

\maketitle


\section{Introduction} \label{sec:intro}

This paper is about specialization properties of polynomials $P(\underline t,\underline y)$ with coefficients in an integral domain $Z$. The $k+n$ variables from the two tuples ${\underline t}=(t_1,\ldots,t_k)$ and ${\underline y} = (y_1,\ldots,y_n)$ ($k,n\geq 1$) are of two types; the $t_i$ are those to be specialized, unlike the $y_i$. The next statement introduces both a central property and a main result of the paper. 

Say that a non-unit $a\in Z$, $a\not=0$, is {\it a fixed divisor of $P$
\hbox{w.r.t.} $\underline t$} if $P(\underline m,\underline y) \equiv 0 \pmod{a}$ for every $\underline m\in Z^k$, and denote the set of all fixed divisors by ${\mathcal F}_{\underline t}(P)$.

\begin{theorem} \label{thm:hilbert-schinzel-main-examples}
Let $Z$ be the ring of integers of a number field of class number $1$ or any polynomial ring $R[x_1,\ldots,x_r]$ ($r\geq 1$) over 
a UFD $R$ \footnote{As usual, UFD stands for Unique Factorization Domain and PID for Principal Ideal Domain.
}.
 Then the ring $Z$ has the {\rm Hilbert-Schinzel specialization property}, for any integers $k,n,s\geq 1$; \hbox{i.e.} the following holds:

\noindent
Let $P_1(\underline t,\underline y), \ldots, P_s(\underline t,\underline y)$ be $s$ polynomials, irreducible in $Z[\underline t,\underline y]$, of degree $\geq 1$ in $\underline y$.  Assume that the product $P_1\cdots P_s$ has no fixed divisor in $Z$ \hbox{w.r.t.} $\underline t$.
Then there is a Zariski-dense subset $H\subset Z^k$ such that $P_1(\underline m,\underline y),\ldots, P_s(\underline m,\underline y)$ are irreducible in $Z\hskip 1pt [\underline y\hskip 1pt]$ for every $\underline m\in H$. 
\end{theorem}

\begin{remark}
The {fixed divisor} assumption ${\mathcal F}_{\underline t}(P_1\cdots P_s)=\emptyset$ is necessary, and may fail. For example, the polynomial $P = (t^p-t)y+ (t^p-t+p)$, with $p$ a prime number, is irreducible in $\Zz[t,y]$; and $p\in {\mathcal F}_{t}(P)$,
since $p$ divides $(m^p-m)$ for every $m\in \Zz$.  A similar example occurs with $Z=\Ff_q[u]$. Take $P= (t^q-t+u)y+ (t^q-t)^2+u$. 
For every $m(u)\in \Ff_q[u]$, the constant term of $m(u)^q - m(u)$ is zero, so $P(m(u),y)$ is divisible by $u$. 
\end{remark}

The name ``Schinzel'' in our specialization property refers to the Schinzel Hypothesis \cite{schinzel-sierpinski}, which corresponds to the case ($k=1$, $n=0$, $Z=\Zz$): if $P_1(t),\ldots,P_s(t)$ are irreducible in $\Zz[t]$ and the product has no fixed prime divisor, then $P_1(m),\ldots,P_s(m)$ are prime numbers for infinitely many $m\in \Zz$.
This statement implies many famous conjectures in number theory, like the Twin Prime conjecture (for $P_1(t)=t$ and $P_2(t)=t+2$). It is however still out of reach; the case $n=0$ is excluded in Theorem \ref{thm:hilbert-schinzel-main-examples}. 

Another special case of interest is when $Z=\Zz$ and 
each polynomial $P_i$ is of the form $P_i = P_{i1}(\underline t) y_1 +\cdots + P_{i\ell}(\underline t)y_\ell$. Theorem \ref{thm:hilbert-schinzel-main-examples} then concludes, under the corresponding assumptions, that for every $\underline m$ in some Zariski-dense subset of $\Zz^k$, the values $P_{i1}(\underline m),\ldots, P_{i\ell}(\underline m)$ are {\it coprime}\footnote{Elements from an integral domain are {\it coprime} if they have no common divisor other than units.}, for each $i=1,\ldots,s$. This was proved by Schinzel \cite{schinzel2002}; see also \cite{ekedahl91} and \cite{poonen2003} for the special case $s=1$, $\ell=2$ but with a positive density result for the good $\underline m$.

This coprime conclusion is interesting for its own sake.
Theorem \ref{thm:hilbert-schinzel-main-examples} already carries it over to more general rings than $\Zz$. We show that it holds on even 
more rings. For simplicity, we restrict below to the situation that {\it one} set of polynomials $P_j(t)$ in {\it one} variable is given, and refer to Theorem 
\ref{thm:main_theorem_gen} for the general version.

\begin{theorem} \label{thm:main_theorem} Assume that $Z$ is a UFD or a Dedekind domain.
Let $Q$ be the fraction field of $Z$.
Then the {\rm coprime Schinzel Hypothesis} 
holds for $Z$, \hbox{i.e.} the following is true: Let $P_1(t),\ldots,P_\ell(t) \in Z[t]$ be $\ell\geq 2$ nonzero po\-ly\-no\-mials, coprime in 
$Q[t]$ and such that: 
\vskip 1,5mm

\noindent
{\rm (AV)} {\it  no non-unit of $Z$ divides all values $P_1(z),\ldots, P_\ell(z)$ with $z\in Z$}.
\vskip 1,5mm

\noindent
Then there exists an element $m\in Z$ such that $P_1(m)$,\ldots,$P_\ell(m)$ are coprime in $Z$.
 \end{theorem}

{Assumption on Values} {\rm (AV)} is the exact translation of the fixed divisor assumption ${\mathcal F}_{t}(P)=\emptyset$ 
for the polynomial $P=P_1(t)y_1+\cdots + P_\ell(t)y_\ell$ considered above.

\begin{remark} \label{rem:PID} (a) The situation that $Z$ is a UFD is the natural context for the coprime Schinzel 
Hypothesis: primes
are the irreducible elements, Gauss's lemma is available, etc. We will however not use the full UFD property and prove Theorem \ref{thm:main_theorem} for domains that we call {\it near UFD}. These play a central role in the paper and are defined by this sole property:
 {\it every non-zero element has finitely many prime divisors, and every non-unit has at least one}; we say more on {near UFDs} in \S \ref{ssec:near_UFD}.
 Theorem \ref{thm:main_theorem} also holds for some non near UFDs, starting with Dedekind domains;
the ring of entire functions is another type of example (Proposition \ref{prop:entire_functions}); on the other hand, the coprime Schinzel Hypothesis may fail, \hbox{e.g.} for $Z=\Zz[\sqrt{5}]$ (Proposition \ref{prop:Zsqrt5}).
\vskip 1mm

\noindent
(b) If $Z$ is infinite, then infinitely many $m$ in fact satisfy the conclusion of Theorem \ref{thm:main_theorem} (see Remark \ref{rem:small_density}). If $Z$ is finite, it is a field, and for fields, ``coprime'' means ``not all zero''. This makes the {\rm coprime Schinzel Hypothesis} obviously true, with the difference for finite fields that the infiniteness of good $m$ is of course not true.
\footnote{Passing from ``at least one" to ``infinitely many" prime values is not nearly as convenient for the original Schinzel Hypothesis. Indeed, \cite{skoro-schinzel} establishes asymptotic results showing that ``most" irreducible integer polynomials without fixed prime divisors take at least one prime value, whereas the infiniteness assertion is not known for a single non-linear polynomial.}
\vskip 0,5mm

\noindent
(c) Theorem \ref{thm:main_theorem} with $Z=\Zz$, contained as we said in \cite[Thm.1]{schinzel2002}, is also a co\-ro\-llary of  \cite[Thm.1.1]{BDN20}, which shows this stronger property  for $Z$ a PID:
\vskip 0,5mm

\noindent
(**) {\it for $P_1(t),\ldots,P_\ell(t)$ as in Theorem \ref{thm:main_theorem}, but not necessarily satisfying assumption {\rm (AV)}, {\it the set ${\mathcal D}=\{\hbox{\rm gcd}(P_1(m),\ldots,P_\ell(m))\hskip 2pt | \hskip 2pt m\in \Zz\}$ is finite and stable under gcd}.  }
\vskip 0,5mm

\noindent
We show in \S \ref{ssec:gcd_stable} that this property is false in general when $Z$ is only a UFD. 
\end{remark}

In addition to the original Schinzel Hypothesis and its coprime version, Theorem \ref{thm:hilbert-schinzel-main-examples} relates to 
Hilbert's Irreducibility Theorem (HIT). In the setup of Theorem \ref{thm:hilbert-schinzel-main-examples} and with $Q$ the fraction field of $Z$, the classical Hilbert result concludes that for every $\underline m$ in some Zariski-dense subset $H\subset Q^k$, the polynomials $P_1(\underline m,\underline y),\ldots, P_s(\underline m,\underline y)$ are irreducible {\it in} $Q[\underline y]$ \cite[Theorem 13.14.2]{FrJa}. In Theorem \ref{thm:hilbert-schinzel-main-examples}, we insist that $H\subset Z^k$ \underline{and} the irreducibility of $P_1(\underline m,\underline y),\ldots, P_s(\underline m,\underline y)$ be {\it over the ring $Z$}, \hbox{i.e.} in $Z[\underline y]$. As $Z$ is a UFD, this is equivalent to $P_1(\underline m,\underline y),\ldots, P_s(\underline m,\underline y)$ being irreducible in $Q[\underline y]$ \underbar{{and}} {\it primitive w.r.t.} $Z$ \footnote{A polynomial over an integral domain $Z$ is {\it primitive \hbox{\it w.r.t.} $Z$} 
if its coefficients are coprime in $Z$. A monomial is primitive iff its leading coefficient is a unit of $Z$. The zero polynomial is not primitive.}.

For an integral domain that is not necessarily a UFD, we generalize the Hilbert-Schinzel property as follows. Assume that $Z$ is of characteristic $0$ or imperfect\footnote{{\it Imperfect} means that $Z^p\not=Z$ if $p={\rm char}(Z)$. This ``imperfectness assumption'' is made to avoid some subtelty from the Hilbertian field theory (\hbox{e.g.} explained  in \cite[\S 4.1]{BDN19}) that otherwise leads to distinguish between Hilbertian fields and strongly Hilbertian fields  and is irrelevant in this paper.}.

\begin{definition} \label{def:hilbert-schinzel}
The ring $Z$ has the {\it Hilbert-Schinzel specialization property} for integers $k,n,s\geq 1$ if the following holds. Let $P_1(\underline t,\underline y), \ldots, P_s(\underline t,\underline y)$ be $s$ \hbox{polynomials}, irreducible in $Q[\underline t,\underline y]$, primitive \hbox{w.r.t.} $Z$, of degree $\geq 1$ in $\underline y$.  Assume that $P_1\cdots P_s$ has no fixed divisor in $Z$ \hbox{w.r.t.} $\underline t$.
Then there is a Zariski-dense subset $H\subset Z^k$ such that \hbox{for every $\underline m \in H$}, the polynomials $P_1(\underline m,\underline y),\ldots, P_s(\underline m,\underline y)$ are irreducible in $Q\hskip 1pt [\underline y\hskip 1pt]$ and primitive \hbox{w.r.t.} $Z$.
\end{definition}

It follows from the conclusion that for $\underline m \in H$, the polynomials $P_1(\underline m,\underline y),\ldots, P_s(\underline m,\underline y)$ are irreducible in $Z\hskip 1pt [\underline y\hskip 1pt]$; this implication holds without the UFD assumption.

More classical definitions (recalled in Definition \ref{def:HilbertUFD}) disregard the primitivity part. For a {\it Hilbertian ring},
only the irreduciblity in $Q\hskip 1pt [\underline y\hskip 1pt]$ is requested in the conclusion, and the fixed divisor condition ${\mathcal F}_{\underline t}(P_1\cdots P_s) = \emptyset$ is not assumed.
If $Z$ is a field (and so conditions on primitivity and fixed divisors  automatically hold and may be omitted), Definition \ref{def:hilbert-schinzel} is that of a {\it Hilbertian field}.

The following result generalizes Theorem \ref{thm:hilbert-schinzel-main-examples}. 

\begin{theorem} \label{thm-intro}
Assume that $Z$ is a Hilbertian ring. Then we have the following.
\vskip 0,5mm

\noindent
{\rm (a)} If $Z$ is a near UFD\footnote{as defined in Remark \ref{rem:PID}(a).}, the Hilbert-Schinzel property holds for any $k,n,s\geq 1$.
\vskip 0,5mm

\noindent
{\rm (b)} If $Z$ is a Dedekind domain, the Hilbert-Schinzel property holds with $k=1$ and $n,s\geq 1$.
\end{theorem}

\vskip -0,8mm

Hilbert's Irreducibility Theorem is one of the few general and powerful tools in \hbox{Arithmetic} Geometry. Typically it is used  
when one needs to find irreducible fibers of some morphism above closed points, defined over some {field}.
 The flagship example, Hilbert's motivation in fact, was the realization of the symmetric group $S_k$ 
as a Galois group over $\Qq$, via the consideration of the morphism $\Aa^k \rightarrow \Aa^k/S_k$, or equivalently, of the generic polynomial $P(\underline t,y) = y^k+t_1y^{k-1}+\cdots+t_k$ of degree $k$  (e.g. \cite[\S 3]{Serre-topics}).
More geometric situations {\it \`a la} Bertini are numerous too,  starting with 
that of an irreducible
family of hy\-per\-sur\-faces $(P(\underline t,\underline y) = 0) \subset \Aa^{n}$ parametrized by ${\underline t}\in \Aa^k$  (see \cite[\S 10.4]{FrJa} for a specific statement of the 
Bertini-Noether theorem).
 Theorem \ref{thm-intro} extends the scope of HIT and its applications to allow working over {\it rings}.
It is for example a good tool when investigating the arithmetic of families of number rings $\Zz[\underline t,y]/\langle P(\underline t,y)\rangle$ with $\underline t \in \Zz^k$, or, in a geometric context, to deal with  \hbox{Bertini irreducibility conclusions 
over rings.}
\vskip 2mm

\begin{remark} (a) \cite[Theorem 4.6]{BDN19} provides a large class of Hilbertian rings: those domains $Z$ such that the fraction field $Q$ has {\it a product formula} (and is of characteristic $0$ or imperfect). We refer to \cite[\S 15.3]{FrJa} or \cite[\S 4.1]{BDN19} for a full definition. The basic example is $Q=\Qq$. The product formula is: $\prod_{p} |a|_p \cdot |a| = 1$ for every $a\in \Qq^\ast$, where $p$ ranges over all prime numbers, $|\cdot|_p$ is the $p$-adic absolute value  and $|\cdot |$ is the standard absolute value. Rational function fields $k(x_1,\ldots,x_r)$ in $r\geq 1$ variables over a field $k$, and finite extensions of fields with the product formula are other examples \cite[\S 15.3]{FrJa}. 
\vskip 1mm

\noindent
(b) The more concrete product formula condition on $Q$ may thus replace the Hilbertian ring assumption in Theorem \ref{thm-intro}. This shows Theorem \ref{thm:hilbert-schinzel-main-examples} as a special case of Theorem \ref{thm-intro}(a). This also provides a large class of rings to which  Theorem \ref{thm-intro}(b) applies: {\it all rings of integers of number fields}. On the other hand, as mentioned earlier, the coprime Schinzel Hypothesis fails for $Z=\Zz[\sqrt{5}]$, and so, so does the Hilbert-Schinzel property. Yet, $\Zz[\sqrt{5}]$ is a Hilbertian ring; 
it is however neither a near UFD nor a Dedekind domain.
\vskip 1mm

\noindent
(c) It is unclear whether Theorem \ref{thm-intro}(b) extends to the situation $k\geq 1$. We refer to Theorem \ref{thm:dedekind-HS-k} for a version of Theorem \ref{thm-intro} for Dedekind domains with $k\geq 1$ and $s=1$.
\vskip 1mm

\noindent
(d) We show further, in Lemma \ref{lem:prelim-proof-R[u]}, that for a near UFD $Z$, assumption ${\mathcal F}_{\underline t}(P_1\cdots P_s) = \emptyset$ always holds in Definition \ref{def:hilbert-schinzel} {\rm (and so can be omitted})  if $Z$ has this {\it infinite residue property}: every principal prime ideal $pZ$ is of infinite norm $|Z/pZ|$. Furthermore this property au\-to\-ma\-tically holds in these cases: (a) $Z=R[u_1,\ldots,u_r]$ is any polynomial ring over an in\-te\-gral domain $R$ unless $Z=\Ff_q[u]$,
(b) if $Z$ contains an infinite field. The infinite residue pro\-pe\-rty  fails if $Z$ is $\Zz$ or, more generally, the ring of integers of a number field. Other ways to get rid of the assumption ${\mathcal F}_{\underline t}(P_1\cdots P_s) = \emptyset$ are explained in \S \ref{get-rid-SC}.
\end{remark}

\noindent
The paper is organized as follows. The coprime Schinzel Hypothesis (from Theorem \ref {thm:main_theorem}) will be defined in its general form for $s\geq 1$ sets of polynomials $\{P_{i1}(\underline t),\ldots ,P_{i\ell_i}(\underline t)\}$ in $k\geq 1$ variables $t_1,\ldots,t_k$ (see Definition \ref{th:schinzel-coprime2}). Section \ref{sec:further} is devoted to the special case $k=s=1$, \hbox{\hbox{i.e.}} the case considered in Theorem \ref {thm:main_theorem}, and Section \ref{ssec:proof-main} to the general case $k,s \geq 1$. The Hilbert-Schinzel specialization property (from Definition \ref{def:hilbert-schinzel}) is discussed in Section \ref{ssec:SH_more_general_rings}; in particular, Theorem \ref{thm-intro} is proved there.

\section{The coprime Schinzel Hypothesis - case $k=s=1$} \label{sec:further}

For simplicity, and to avoid confusion, we denote by \CopSch\hskip 2pt  the coprime Schinzel Hypothesis in the  form given in Theorem \ref{thm:main_theorem}  (which corresponds to the case $k=s=1$ of Definition \ref{th:schinzel-coprime2} given later).

In \S \ref{ssec:preliminaries}, we introduce a basic parameter of the problem. We then prove Theorem \ref{thm:main_theorem} for Dedekind domains in \S \ref{ss:ded}. The other case of Theorem \ref{thm:main_theorem} will be proved in the more general situation $k,s \geq 1$ in \S \ref{ssec:proof-main} for {\it near UFDs}. We introduce them and briefly discuss some basic properties in \S \ref{ssec:near_UFD}.
In \S \ref{sec:section 3}, we consider property \CopSch\hskip 2pt over rings that are neither near UFDs nor Dedekind domains.  
Finally \S \ref{ssec:gcd_stable} discusses the gcd stability property mentioned in Remark \ref{rem:PID}(c) and displays the counter-example announced there.
\vskip 1mm

Let $Z$ be an integral domain. Denote the fraction field of $Z$ by $Q$ and the group of invertible elements, also called units, by $Z^\times$.

\subsection{A preliminary lemma} \label{ssec:preliminaries}
Let $t$ be a variable and $P_1,\ldots, P_\ell\in Z[t]$ be $\ell$ nonzero polynomials  ($\ell\geq 2)$, assumed to be coprime in $Q[t]$; equivalently, they have no common root in an algebraic closure of $Q$.
As $Q[t]$ is a PID, we have $\sum_{i=1}^\ell   P_i\hskip 2pt Q[t]= Q[t]$. It follows that 
$(\sum_{i=1}^\ell   P_i\hskip 2pt Z[t]) \cap Z$ is a nonzero ideal of $Z$. Fix a nonzero element $\delta\in Z$ 
in this ideal. For example, if $\ell=2$, one can take $\delta$ to be the resultant $\rho={\rm Res}(P_1,P_2)$
\cite[V \S 10]{Langoriginal}.

\begin{lemma} \label{lemma:1} 
For every $m\in Z$, denote the set of common divisors of $P_1(m),\ldots,P_\ell(m)$
by ${\mathcal D}_m$. Then for every $m\in Z$, the set  ${\mathcal D}_m$ is a subset of the set of divisors of $\delta$. Furthermore,  for every $z \in Z$, we have ${\mathcal D}_m = {\mathcal D}_{m+z \delta}$.
\end{lemma}

\begin{proof}
From the coprimality assumption in $Q[t]$ of $P_1,\ldots, P_\ell$, there exist some polynomials $V_1,\ldots,V_\ell \in Z[t]$ satisfying a B\'ezout condition
\vskip 1mm

\noindent
\centerline{$V_1(t)P_1(t)+\cdots+V_\ell(t)P_\ell(t)=\delta.$}
\vskip 1mm

\noindent
The same holds with any $m\in Z$ substituted for $t$. The first claim follows. For the second claim, we adjust an argument of Frenkel and Pelik\'{a}n \cite{FP} who observed this periodicity
property in the special case ($\ell=2$, $Z=\Zz$) with $\delta$ equal to the resultant $\rho={\rm Res}(P_1,P_2)$.

For every $m,z \in Z$, we have $P_i(m+z \delta) \equiv P_i(m) \pmod {\delta}$, $i=1,\ldots,\ell$. It follows that the common divisors of $P_1(m),\ldots,P_\ell(m),\delta$ are the same as those of $P_1(m+z \delta)$,$\ldots$, $P_\ell(m+z \delta), \delta$. As both the common divisors of $P_1(m),\ldots,P_\ell(m)$ and those of $P_1(m+z \delta)$, $\ldots$, $P_\ell(m+z \delta)$ divide $\delta$, the conclusion ${\mathcal D}_m = {\mathcal D}_{m+z\delta}$ follows.
\end{proof}

\begin{remark}[{\rm on the set of ``good'' $m$}] \label{rem:small_density}
It follows from the periodicity property that, if $Z$ is infinite, then the set, say ${\mathcal S}$, of all $m\in Z$ such that $P_1(m)$,\ldots,$P_\ell(m)$ are coprime in $Z$, is infinite if it is nonempty. 
The set ${\mathcal S}$ can nevertheless be of arbitrarily small density. Take $Z=\Zz$, $P_1(t) = t$, $P_2(t)=t+\Pi_h$, with $\Pi_h$ ($h\in \Nn$) the product of primes in $[1,h]$. The set ${\mathcal S}$ consists of the integers which are prime to $\Pi_h$. Its density is: 

\vskip 1mm
\centerline{$\displaystyle \frac{\varphi(\Pi_h)}{\Pi_h} = \left(1-\frac{1}{2}\right) \cdots \left(1-\frac{1}{p_h}\right)$}
\vskip 1mm

\noindent
where $p_h$ is the $h$-th prime number and $\varphi$ is the Euler function. The sequence $\varphi(\Pi_h)/\Pi_h$
tends to $0$ when $h\rightarrow \infty$ (since the series $\sum_{h=0}^\infty 1/p_h$ diverges).
\end{remark}

\subsection{Proof of Theorem \ref{thm:main_theorem} for Dedekind domains}
\label{ss:ded}
The assumptions and notation of this paragraph, including Proposition \ref{rem:plus} below are as follows. The ring $Z$ is a Dedekind domain. As in the coprime Schinzel Hypothesis, $P_1(t),\dots, P_\ell(t)$ are $\ell \geq 2$ nonzero polynomials in $Z[t]$, coprime in $Q[t]$ and satisfying Assumption (AV); $\delta$ is the associated parameter from \S \ref{ssec:preliminaries} and $\delta Z = \prod_{i=1}^r \mathcal{Q}_i^{e_i}$ is the factorization of the principal ideal $\delta Z$ into prime ideals of $Z$.  
We also define $\mathfrak{I}$ as the ideal generated by all values $P_1(z),\dots, P_\ell(z)$ with $z\in Z$, and factor $\mathfrak{I}$ into prime ideals: $\mathfrak{I} = \prod_{i=1}^q \mathcal{Q}_i^{g_i}$; we may assume  that each of  the prime ideals $\mathcal{Q}_1,\ldots,\mathcal{Q}_r$ dividing $\delta Z$ indeed occurs in the product by allowing exponents $g_i$ to be $0$.
\vskip 1mm

Consider then the ideals $\mathcal{Q}_i^{g_i+1}$,  $i=1,\ldots,r$. Either $r\leq 1$ or any two of them are comaximal\footnote{Two ideals $U,V$ of an integral domain $Z$ are {\it comaximal} if $U+V = Z$.}. As none of them contains $\mathfrak{I}$, for each $j=1,\ldots,r$, there exists $i_j\in \{1,\dots, \ell\}$ and $m_j\in Z$ such that $P_{i_j}(m_j)\not\equiv 0 \pmod{\mathcal{Q}_j^{g_j+1}}$. 
 The Chinese Remainder Theorem yields an element $m\in Z$ such that $m\equiv m_j \pmod{\mathcal{Q}_j^{g_j+1}}$, for each $j=1,\ldots,r$. 
It follows that $P_{i_j}(m)\notin \mathcal{Q}_j^{g_j+1}$, and so
\vskip 1mm

\noindent
\centerline{$(P_1(m),\ldots,P_\ell(m)) \not\equiv (0,\ldots,0) \pmod {\mathcal{Q}_i^{g_i+1}},\hskip 1mm \hbox{for each}\ i=1,\ldots, r$.}
\vskip 1,5mm

Proposition \ref{rem:plus}(a) below (with $z=0$ and $\alpha=m$) concludes that $P_1(m),\dots, P_\ell(m)$ 
are coprime in $Z$, 
thus ending the proof of Theorem \ref{thm:main_theorem} for Dedekind domains. The more general statement in Proposition \ref{rem:plus}(a) and the additional statement (b) will be used later.

\begin{proposition} \label{rem:plus} Under the assumption and notation of \S \ref{ss:ded},
let $\omega \in Z$ be a multiple of $\delta$ 
and let $\alpha\in Z$ be an element such that 
\vskip 0,5mm

\noindent
{\rm (*)} \hskip 15mm $(P_1(\alpha),\ldots,P_\ell(\alpha)) \not\equiv (0,\ldots,0) \pmod {\mathcal{Q}_i^{g_i+1}},\hskip 1mm i=1,\ldots, r$.
\vskip 1mm

\vskip 0,5mm

\noindent
{\rm (a)} Then, for every $z \in Z$, the elements $P_1(\alpha+z \omega),\dots, P_\ell(\alpha+z \omega)$ 
are coprime in $Z$.
\vskip 1mm

\noindent
{\rm (b)} Furthermore, if instead of {\rm (AV)}, it is assumed that {\it no \underbar{prime} of $Z$ divides all values $P_1(z),\ldots,P_\ell(z)$ with $z\in Z$}, then for every $z \in Z$, the elements $P_1(\alpha+z \omega),\dots, P_\ell(\alpha+z \omega)$ 
have no common \underbar{prime} divisor. 
\end{proposition}

\begin{proof} (a) From Lemma \ref{lemma:1}, it suffices to show that $P_1(\alpha),\dots, P_\ell(\alpha)$ have no non-unit divisors. Assume on the contrary that 
\vskip 0,5mm

\noindent
\centerline{$(P_1(\alpha),\dots, P_\ell(\alpha)) \equiv 0 \pmod{a}$ for some non-unit $a\in Z$.}
\vskip 0,5mm

\noindent
By the definition of $\delta$ and Lemma \ref{lemma:1}, the element $a$ divides $\delta$. Thus 
the prime ideal factorization of $aZ$ is of the form $aZ=\prod_{i=1}^r \mathcal{Q}_i^{f_i}$, with exponents $f_i\le e_i$. Due to assumption (AV), there must exist an index $i\in \{1,\dots, r\}$ such that $f_{i} > g_{i}$: otherwise $\mathfrak{I} \subset aZ$, \hbox{i.e.} $a$ divides all values $P_1(z),\dots, P_\ell(z)$ ($z\in Z$). Consequenty, $f_i \geq g_{i+1}$ and so $aZ\subset \mathcal{Q}_i^{g_i+1}$, for the same index $i$. But this contradicts assumption (*).

\vskip 1mm

\noindent
(b) Merely replace in the proof of (a) the non-unit $a$ by a prime $p$, and resort to the variant of (AV) assumed in (b).
\end{proof}

\subsection{A few words on near UFDs} \label{ssec:near_UFD} 
This subsection says more on near UFDs for which we will prove the full coprime Schinzel Hypothesis in \S \ref{ssec:proof-main}. They will also serve, with Dedekind domains, as landmarks in the discussion of the Hypothesis over other domains in \S \ref{sec:section 3}.

Recall from Remark \ref{rem:PID} that we call an integral domain $Z$ a {\it near UFD} if  
{every non-zero element has finitely many prime divisors, and every non-unit has at least one.}
Of course, a UFD is a near UFD. A simple example showing that the converse does not hold is the ring $\mathbb{Z}_p + X\mathbb{Q}_p[X]$ of polynomials over $\mathbb{Q}_p$ with constant coefficient
in $\mathbb{Z}_p$; 
see \hbox{Example \ref{rem:other-examples}.}

It is worth noting further that 
\vskip 0,5mm

\noindent
(a) as for UFDs, every irreducible element $a$ of a near UFD $Z$ is a prime: indeed, such an $a$ is divisible by a prime $p$ of $Z$; being irreducible,
$a$ must in fact be associate to $p$.
\vskip 0,5mm

\noindent
(b) unlike UFDs, near UFDs do not satisfy the Ascending Chain Condition on Principal Ideals in general, \hbox{i.e.} there exist near UFDs which 
have an infinite strictly ascending chain of principal ideals; see Example \ref{rem:other-examples} below.
\vskip 0,5mm

It is classical that being a UFD is equivalent to satisfying these two conditions: the Ascending Chain Condition on Principal Ideals holds and 
every irreducible is a prime. Thus a near UFD is a UFD if and only if it satisfies the Ascending Chain Condition on Principal Ideals. 
In particular, a near UFD that is Noetherian is a UFD.

\begin{example} \label{rem:other-examples}
%
Let $R$ be a UFD, not a field, and let $K$ the field of fractions of $R$. Let $Z=R+XK[X]$ be the ring of polynomials over $K[X]$ with 
constant coefficient in $R$. Then it is well-known that $Z$ is not a UFD. Indeed, any prime element $p$ of $R$ remains prime in $Z$, and one has factorizations $X = (X/p)\cdot p = (X/p^2)\cdot p^2$, corresponding to an infinite ascending chain $(X) \subset (X/p) \subset (X/p^2)\subset\dots$ . Moreover, all irreducible elements of $Z$ are prime, and the non-constant prime elements are (up to associates) exactly the non-constant polynomials with constant coefficient $1$. Clearly, every non-constant element of $Z$ has at least one, but finitely many prime divisors of this kind. Assume now additionally that $R$ has only finitely many non-associate prime elements. Then every non-constant polynomial in $Z$ has only finitely many prime divisors altogether, and the same holds more obviously for (non-zero) constants. Therefore, $Z$ is a near UFD in this case.
\end{example}

\begin{remark}[a further advantage of near UFDs]In near UFDs, a non-unit is always divisible by a prime. This is not the case in general.
For example $a=6$ in the ring $\mathbb{Z}[\sqrt{-5}]$ does have irreducible divisors but no prime divisors; or $a=2$ in the ring of algebraic integers does not even have any irreducible divisors (the ring has no irreducibles). Pick an element $a\in Z$ as in these two examples. The polynomials $P_1(y)=ay$ and $P_2(y)=a(y+1)$ are coprime in $Q[y]$ and a common prime divisor of all values of $P_1$ and $P_2$ would have to divide $a$ and therefore does not exist. Yet there is no $m\in Z$ such that $P_1(m)$ and $P_2(m)$ are coprime. To avoid such examples, we insist in our assumption (AV) that all elements $P_1(m), \ldots, P_\ell(m)$ with $m\in Z$ be coprime, and not just that no prime divides them all. This subtelty vanishes of course if $Z$ is a near UFD.
\end{remark}

\subsection{Other domains} \label{sec:section 3}
While property \CopSch\hskip 2pt is completely well behaved in the class of Dedekind domains
and, as we will see in \S \ref{ssec:proof-main}, in that of near UFDs, we show in this subsection that 
the behavior inside other classes is rather erratic.  
For example, we produce a non-Noetherian B\'ezout domain\footnote{Recall that a domain is called {\it B\'ezout} if any two elements have a greatest common divisor which is a linear combination of them. Equivalently, the sum of any two principal ideals is a principal ideal.} for which \CopSch\hskip 2pt holds (Proposition \ref{prop:entire_functions}) and another one for which it does not (Remark \ref{rem:FCopSch}). 
We also show that \CopSch\hskip 2pt fails over certain number rings, such as the domain $\Zz[\sqrt{5}]$ (Proposition \ref{prop:Zsqrt5}).

\subsubsection{Non-Noetherian domains satisfying {\rm \CopSch}} 
Proposition \ref{prop:entire_functions} shows a ring that is not a near UFD but satisfies {\rm \CopSch}.
Proposition \ref{prop:Z_p_bar} even produces a domain $Z$ that fulfills \CopSch\hskip 2pt even though $Z$ has non-units not divisible by any prime.

\begin{proposition} \label{prop:entire_functions}
 The ring $Z$ of entire functions is a B\'ezout domain which satisfies {\rm \CopSch}, but is not Noetherian and is not a near UFD.
\end{proposition}

\begin{proof}
The ring $Z$ is a B\'ezout domain (see \hbox{e.g.} \cite{Cohn}) whose prime elements, up to multiplication with units, are exactly the linear polynomials $x-c$ ($c\in \mathbb{C}$); indeed, an element of $Z$ is a non-unit if and only if it has a zero in $\mathbb{C}$, and an element with a zero $c$ is divisible by $x-c$ due to Riemann's theorem on removable singularities. In particular, every non-unit of $Z$ has at least one prime divisor. However the set of zeroes of a nonzero entire function may be infinite. This shows that the ring $Z$ is not a near UFD, and is not Noetherian either. Note also that existence of a common prime divisor for a set of elements of $Z$ is equivalent to existence of a common root in $\mathbb{C}$.

Consider now finitely many polynomials $P_1(t), \dots, P_\ell(t)\in Z[t]$ which are coprime in $Q[t]$ and for which no prime of $Z$ divides all values $P_i(m)$ ($i=1,\dots, \ell$; $m\in Z$). For each $P_i$, we may factor out the gcd $d_i$ of all coefficients and write $P_i(t) = d_i\tilde{P_i}(t)$, where the coefficients of $\tilde{P_i}$ are coprime. 
Then $d_1,\dots, d_\ell$ are necessarily coprime (since their common prime divisors would divide all values $P_i(m)$ ($i=1,\dots, \ell$; $m\in Z$)). 
We now consider specialization of $\tilde{P_1},\ldots, \tilde{P_\ell}$ at {\it constant} functions $m\in \mathbb{C}$. 

Recall that by Lemma \ref{lemma:1}, there exists a nonzero $\delta \in Z$ such that for every $m\in Z$, and in particular for every $m\in \mathbb{C}$, every common divisor  of $P_1(m),\dots, P_\ell(m)$ is a divisor of $\delta$. The entire function $\delta$ has a countable set $S$ of zeroes. To prove {\rm \CopSch}, it suffices to find $m\in \mathbb{C}$ for which no $z\in S$ is a zero of all of $P_1(m),\ldots,P_\ell(m)$. 
For $i\in \{1,\dots, \ell\}$, denote by $S_i$ the set of all $z\in S$ which are not a root of $d_i$. Since $d_1,\dots, d_\ell$ are coprime, one has $\cup_{i=1}^\ell S_i = S$.
Now fix $i$ for the moment, and write $\tilde{P_i}(t) = \sum_{j=0}^k (\sum_{k=0}^\infty a_{jk} x^k) t^j$ with $a_{jk}\in \mathbb{C}$, via power series expansion of the coefficients of $\tilde{P_i}$.
Let $z\in S_i$. Since $z$ is not a root of all coefficients of $\tilde{P_i}$, evaluation $x\mapsto z$ yields a nonzero polynomial, which thus has a root at only finitely many values $t\mapsto m\in \mathbb{C}$. This is true for all $z\in S_i$, whence the set of $m\in \mathbb{C}$ such that $P_i(m) = d_i\tilde{P_i}(m)$ has a root at {\it some} $z\in S_i$ is a countable set. In total, the set of all $m\in \mathbb{C}$ such that ${P_1}(m),\dots, {P_\ell}(m)$ have a common root in $S$ (and hence in some $S_i$) is countable as well. Choose $m$ in the complement of this set to obtain the assertion.
\end{proof}

\begin{proposition} \label{prop:Z_p_bar}
Let $Z= \overline{\Zz_p}$ be the integral closure of $\mathbb{Z}_p$ in $\overline{\mathbb{Q}_p}$. Then $Z$ has non-units that are not divisible by any prime and satisfies {\rm \CopSch}.  
\end{proposition}

\begin{proof} The domain $Z$ is a (non-Noetherian) valuation ring, whose nonzero finitely 
generated ideals are exactly the principal ideals $p^rZ$ with $r$ a non-negative rational number. 
Prime elements do not exist in $Z$, whence the first part of the assertion.

Now take finitely many nonzero polynomials $P_1(t),\dots, P_\ell(t)\in Z[t]$, coprime in $\overline{\mathbb{Q}_p}[t]$, and assume that all values $P_i(m)$ ($i=1,\dots, \ell$; $m\in Z$) are coprime. Then the coefficients of $P_1,\dots, P_\ell$ must be coprime, and since $Z$ is a valuation ring (\hbox{i.e.}, its ideals are totally ordered by inclusion), one of these coefficients, say the coefficient of $t^d$, $d\ge 0$, of the polynomial $P_1$, must be a unit. We will proceed to show that there exists $m\in Z$ such that $P_1(m)$ is a unit, which will clearly prove \CopSch. To that end we draw the Newton polygon (with respect to the $p$-adic valuation) of the polynomial $P_1(t)-u$, with $u\in Z$ a unit to be specified. Since any non-increasing slope in this Newton polygon corresponds to a set of roots of $P_1(t)-u$ of non-negative valuation (\hbox{i.e.} roots contained in $Z$), it suffices to choose $u$ such that there exists at least one segment of non-increasing slope (see, \hbox{e.g.} Proposition II.6.3 in \cite{neukirch-book} for the aforementioned property of the Newton polygon). This is trivially the case if $d>0$, so we may assume that the constant coefficient of $P_1$ is the only one of valuation $0$. But then simply choose $u=P_1(0)$ and $m=0$.
\end{proof}

\subsubsection{Rings not satisfying {\rm \CopSch}} \label{ssec:Z[sqrt5]}

\begin{proposition}
\label{prop:local}
Let $Z$ be a domain and $p,q$ be non-associate irreducible elements of $Z$ such that $Z/(pZ \cap qZ)$ is a finite local ring.
Then  {\rm \CopSch}\hskip 2pt fails for $Z$.
\end{proposition}

The proof rests on the following elementary fact.
\begin{lemma}
\label{lem:twovalues}
Let $R$ be a finite local ring and let $a,b\in R$ be two distinct elements. Then there exists a polynomial $f\in R[t]$ taking the value $a$ exactly on the units of $R$, and the value $b$ everywhere else.
\end{lemma}
\begin{proof}
Start with $a=1$ and $b=0$.  The unique maximal ideal of a finite local ring is necessarily the nilradical, \hbox{i.e.}, every non-unit of $R$ is nilpotent. In particular, there exists $n\in \mathbb{N}$ (\hbox{e.g.} any sufficiently large multiple of $|R^\times|$) such that $r^n=0$ for all nilpotent elements $r\in R$, and $r^n=1$ for all units $r$.
Setting $f(t)=t^n$ finishes the proof for $a=1$, $b=0$. The general case then follows by simply setting $f(t)=(a-b)t^n+b$.
\end{proof}

\begin{proof}[Proof of Proposition \ref{prop:local}]
Set $J=pZ \cap qZ$. As already used above, locality of $Z/J$ implies that all non-units of $Z/J$ are nilpotent. In particular, $p,q\notin J$, but there exist $m,n\in \mathbb{N}$ such that $p^m \in J$ and $q^n\in J$. By Lemma \ref{lem:twovalues}, there exist polynomials $P_1, P_2\in Z[t]$ such that $P_1(z)\in \begin{cases}p+J, \ \text{ for } z+J\in (Z/J)^\times \\ J, \ \text{ otherwise } \end{cases}$, and $P_2(z)\in \begin{cases} J, \ \text{ for } z+J\in (Z/J)^\times \\ q+J, \ \text{ otherwise } \end{cases}$. By adding suitable constant terms in $J$ to $P_1$ and $P_2$, one may additionally demand that $p\in P_1(Z)$ and $q\in P_2(Z)$. In particular, $P_1$ and $P_2$ satisfy assumption (AV).
However, by construction, $p$ divides both $P_1(z)$ and $P_2(z)$ for all $z$ such that $z+J\in (Z/J)^\times$, and $q$ divides $P_1(z)$ and $P_2(z)$ for all other $z$, showing that \CopSch\hskip 2pt is not satisfied.
\end{proof}

\begin{proposition} \label{prop:Zsqrt5}
 {\rm \CopSch}\hskip 2pt does not hold for $Z=\mathbb{Z}[\sqrt{5}]$.
\end{proposition}

\begin{proof}
Set $\sigma=\sqrt{5}+1$, so that $Z=\mathbb{Z}[\sigma]$. Note the factorization $2\cdot 2 = \sigma(\sigma-2)$ in $Z$, in which $2$ and $\sigma$ are non-associate irreducible elements. Due to Proposition \ref{prop:local}, it suffices to verify that $Z/(2Z\cap \sigma Z)$ is a local ring. However, since $4, 2\sigma$ and $\sigma^2$ are all in $2Z\cap \sigma Z$, the set $\{0,1,2,3, \sigma,\sigma+1,\sigma+2,\sigma+3\}$ is a full set of coset representatives, and the non-units (\hbox{i.e.} $0,2,\sigma,\sigma+2$) are exactly the nilpotents in $Z/(2Z\cap \sigma Z)$. These therefore form the unique maximal ideal, ending the proof.
\end{proof}

\begin{remark}[The {\it finite} coprime Schinzel Hypothesis] \label{rem:FCopSch}
(a) The proof above of Proposition \ref{prop:Zsqrt5} even shows that a weaker variant of \CopSch\hskip 2pt fails over $\mathbb{Z}[\sqrt{5}]$, namely the variant, say (FinCopSch), for which the exact same conclusion holds but under the following stronger assumption on values:
\vskip 0,5mm

\noindent
(Fin-AV) {\it The set of values $\{P_i(m) \hskip 2pt | \hskip 2pt m\in Z, i=1,\ldots,\ell \}$ contains a finite subset
whose elements are coprime in $Z$}.

\vskip 1mm

\noindent
(b) Here is an example of a domain fulfilling (FinCopSch) but not \CopSch.
Consider the domain $Z=\mathbb{Z}_p[\hskip 1pt p^{p^{-n}}\mid n\in \mathbb{N}]$. 
This is a (non-Noetherian) valuation ring, and in particular a B\'ezout domain (the finitely generated non-trivial ideals are exactly the principal ideals $(p^{m/p^n})$ with $m,n\in \mathbb{N}$).

Property \CopSch\hskip 2pt does not hold for $Z$. Take $P_1(t) = pt$ and $P_2(t) = t^p-t+p$. These are coprime in $Q[t]$. Furthermore $P_1$ and $P_2$ satisfy assumption (AV)
--- indeed, any common divisor certainly divides $p=P_1(1)$ as well as $m(m^{p-1}-1)$ for every $m\in Z$; choosing $m$ in the sequence $(p^{p^{-n}})_{n\in \mathbb{N}}$ shows the claim.
But note that every $m\in Z$ lies inside some ring $Z_0=\Zz_p[p^{p^{-n}}]$ (for a suitable $n$), and the unique maximal ideal of that ring has residue field $\mathbb{F}_p$, meaning that it necessarily contains $m(m^{p-1}-1)$. 
Thus $m(m^{p-1}-1)$ is at least divisible by $p^{p^{-n}}$, which is thus a common divisor of $P_1(m)$ and $P_2(m)$.

On the other hand, $Z$ fulfills property (FinCopSch). Indeed, if $P_1(t),\dots, P_\ell(t)$ are polynomials in $Z[t]$ for which finitely many elements $P_i(m)$ with $m\in Z, i\in \{1,\dots, \ell\}$
exist that are coprime, then automatically one of those values must be a unit (since any finite set of non-units has a suitably high root of $p$ as a common divisor!), yielding an $m$ for which $P_1(m),\dots, P_\ell(m)$ are coprime. 
\end{remark}

\subsection{A UFD not satisfying the gcd stability property} 
\label{ssec:gcd_stable}

Recall, for $Z=\Zz$, the following result from \cite{BDN20} already mentioned in the introduction.
\begin{theorem}
\label{thm:gcdstable}
Let $P_1,\ldots,P_\ell \in Z[t]$ be $\ell\geq 2$ nonzero po\-ly\-no\-mials, coprime in 
$Q[t]$. Then the set ${\mathcal D}=\big\{ \gcd (P_1(m),\ldots,P_\ell(m)) \mid m\in Z \big\}$
is finite and stable under gcd.
\end{theorem}

The proof is given for $Z=\Zz$ in \cite{BDN20} but is valid for any PID.
Theorem \ref{thm:gcdstable} implies property \CopSch.
Indeed, asssumption (AV) exactly means that the gcd of elements of ${\mathcal D}$ is $1$.  
By Theorem \ref{thm:gcdstable},  $\mathcal{D}$ is finite and stable by gcd. Therefore $1 \in \mathcal{D}$, \hbox{i.e.} there exists $m\in Z$ such that $P_1(m),\ldots,P_\ell(m)$ are coprime.
The stability property however cannot be extended to all UFDs.

\begin{example}[a counter-example to Theorem \ref{thm:gcdstable} for the UFD \hbox{$Z=\Zz[x,y,z]$}]
Let 
$$P_1(t) = \big(x^2y^2z + t^2\big)\big(x^2yz^2 + (t-1)^2\big) \in Z[t]$$
$$P_2(t) = \big(xy^2z^2 + t^2\big)\big(x^2y^2z^2 + (t-1)^2\big) \in Z[t]$$

\noindent
These nonzero polynomials are coprime in $Q[t]$: they have no common root in $\overline Q$.

We prove next that the set $\mathcal{D} = \{ \gcd(P_1(m),P_2(m)) \mid m \in Z\}$
is not stable by gcd. Set
\vskip 1mm

\centerline{$
\left\{\begin{matrix}
& d_0 = \gcd( P_1(0), P_2(0) ) = \gcd( x^2y^2z, xy^2z^2 ) = xy^2z \hfill \\
& d_1 = \gcd( P_1(1), P_2(1) ) = \gcd( x^2yz^2, x^2y^2z^2 ) = x^2yz^2 \hfill \\
\end{matrix}\right.
$}
\vskip 1mm

\noindent
and $d = \gcd(d_0,d_1) = xyz$. We prove below that $d \notin \mathcal{D}$.

By contradiction, assume that $d= \gcd (P_1(m),P_2(m))$ for some $m = m(x,y,z) \in Z$.
As $xyz | P_1(m)$ it follows that $xyz | \big(x^2y^2z + m^2\big)\big(x^2yz^2 + (m-1)^2\big)$
whence $xyz | m^2(m-1)^2$ and $xyz | m(m-1)$.
We claim that the last two divisibilities imply that $xyz | m$ or $xyz | m-1$.

Namely, if for instance we had $xy|m$ and $z|m-1$ then, on the one hand, we would have 
$m =xy \hskip 1pt m^\prime$ for some $m^\prime \in Z$ and so $m(0,0,0)=0$, but
on the other hand, we would have $m(x,y,z)-1=z \hskip 1pt m^{\prime\prime}$ for some $m^{\prime\prime}\in Z$ and so $m(0,0,0)=1$. Whence the claim.

Now if $xyz | m$, then $x^2y^2z^2 | m^2$. But then $x^2y^2z | P_1(m)$ and $xy^2z^2 | P_2(m)$, whence $xy^2z | d$, a contradiction. 
The other case for which $xyz | m-1$ is handled similarly.
\end{example}

 \section{The coprime Schinzel Hypothesis - general case}  \label{ssec:proof-main}
 
 The fully general coprime Schinzel Hypothesis and the almost equivalent Primitive Specialization 
 Hypothesis are introduced in \S \ref{ssec:generalization}. Our main result on them, Theorem \ref{thm:main_theorem_gen}, is stated in \S \ref{ssec:main-conclusions}.
Sections \ref {ssec:chinese} and \ref{ssec:reduction} show three lemmas needed for the proof, which is given in \S \ref{ssec:proof_k=1}.
We start in \S \ref{ssec:two_Schinzel} with some observations on fixed divisors. 
\vskip 1mm

Fix an arbitrary integral domain $Z$. 
 
 \subsection{Fixed divisors} \label{ssec:two_Schinzel} 
We refer to \S \ref{sec:intro} for the definition of ``fixed divisor \hbox{w.r.t.} $\underline t$'' of a polynomial $P(\underline t,\underline y)$ and for the associated notation
${\mathcal F}_{\underline t}(P)$.

\begin{lemma} \label{lem:prelim-proof-R[u]} Let $Z$ be an integral domain.
\vskip 0,5mm

\noindent
{\rm (a)} Let $P\in Z[\underline t,\underline y]$ be a nonzero polynomial and $p$ be a prime of $Z$ not dividing $P$. If $p$ is in the set ${\mathcal F}_{\underline t}(P)$ of fixed divisors of $P$, it is of norm $|Z/pZ| \leq \max_{i=1,\ldots,k} \deg_{t_i}(P)$.
\vskip 0,5mm

\noindent
{\rm (b)} Assume that $Z$ is a near UFD and every prime of $Z$ is of infinite norm. Then, for any polynomials $P_1(\underline t,\underline y),\ldots, P_s(\underline t,\underline y)\in Z[\underline t,\underline y]$, primitive \hbox{w.r.t.} $Z$, we have ${\mathcal F}_{\underline t}(P_1\cdots P_s)=\emptyset$.

\vskip 0,5mm

\noindent
{\rm (c)} If $Z=R[u_1,\ldots,u_r]$ is a polynomial ring over an integral domain $R$, and if either 
$R$ is infinite or $r\geq 2$, then every prime $p\in Z$ is of infinite norm. The
same conclusion holds if $Z$ is an integral domain containing an infinite field.
\end{lemma}

On the other hand, $\Zz$ and $\Ff_q[x]$ are typical examples of rings that have primes of finite norm. As already noted, (b) is in fact false for these two rings.

\begin{proof}
(a) is classical. If $p$ is a prime of $Z$ such that $P$ is nonzero modulo $p$ and $|Z/pZ| > \max_{i=1,\ldots,k} \deg_{t_i}(P)$, there exists $\underline m \in Z^k$ such that $P(\underline m,\underline y) \not\equiv 0 \pmod{p}$, \hbox{i.e.} $p\notin {\mathcal F}_{\underline t}(P)$.
\vskip 1mm

\noindent
(b) Assume that the set ${\mathcal F}_{\underline t}(P_1\cdots P_s)$ contains a non-unit $a\in Z$. As $Z$ is a near UFD, one may assume that $a$ is a prime of $Z$. It follows from 
$P_1,\ldots, P_s$ being primitive \hbox{w.r.t.} $Z$ that the product $P_1\cdots P_s$ is nonzero modulo $a$. From (a), the norm $|Z/aZ|$ should be finite, whence a contradiction. Conclude that the set ${\mathcal F}_{\underline t}(P_1\cdots P_s)$ is empty.
\vskip 1mm

\noindent
(c) With $Z=R[u_1,\ldots,u_r]$ as in the statement, assume first that $R$ is infinite. Let $p\in R[u_1,\ldots,u_r]$ be a prime element. Suppose first that $p\not\in R$, say $d=\deg_{u_r}(p)\geq 1$. 
The elements $1,u_r,\ldots,u_r^{d-1}$ are $R[u_1,\ldots,u_{r-1}]$-linearly independent in the integral domain $Z/pZ$.
As $R$ is infinite, the elements $\sum_{i=0}^{d-1} p_i u_r^i$ with $p_0,\ldots,p_{d-1}\in R$ are  infinitely many different elements in $Z/p Z$. Thus $Z/pZ$ is infinite. 
In the case that $p\in R$, the quotient ring $Z/pZ$ is $(R/pR)[u_1,\ldots,u_r]$, which is infinite too. 

If $Z=R[u_1,\ldots,u_r]$ with $r\geq 2$, write $R[u_1,\ldots,u_r] = R[u_1][u_2,\ldots,u_r]$ and use the previous
paragraph with $R$ taken to be the infinite ring $R[u_1]$.

If $Z$ is an integral domain containing an infinite field $k$, the containment $k\subset Z$ in\-du\-ces an injective morphism $k\hookrightarrow Z/pZ$ for every prime $p$ of $Z$. The last claim follows.
\end{proof}

\subsection{The two Hypotheses} \label{ssec:generalization}
Definition \ref{th:schinzel-coprime2} introduces the coprime Schinzel Hypothesis in the general situation of $s$ sets of polynomials in $k$ variables $t_i$, with $k,s\geq 1$. 
The initial definition from Theorem \ref{thm:main_theorem} (denoted \CopSch\hskip 2pt in \S \ref{sec:further}) corresponds to the special case $s=k=1$;
 in particular assumption (AV) from there is  (AV)${}_{\underline t}$ below with $\underline t = t$ and $s=1$.

\begin{definition}
\label{th:schinzel-coprime2}
Given an integral domain $Z$ of fraction field $Q$, we say that the {\it coprime Schinzel Hypothesis holds for $Z$} if  for any integers $k,s \geq 1$, the following property is true:
\vskip 0,5mm

\noindent
\hskip 1pt ${\mathcal CopSch}(k,s)$: Consider $s$ sets $\{P_{11},\ldots,P_{1\ell_1}\}$,$\ldots$, $\{P_{s1},\ldots,P_{s\ell_s}\}$ of nonzero polynomials  in $Z[\underline t]$ (with $\underline t=(t_1,\ldots,t_k)$) such that $\ell_i \geq 2$ and $P_{i1},\ldots,P_{i\ell_i}$ are coprime in $Q[\underline t]$, for each $i=1,\ldots, s$. Assume further that
\vskip 1mm

\noindent
(AV)${}_{\underline t}$ {\it for every non-unit $a\in Z$, there exists $\underline m\in Z^k$, such that, for each $i=1,\ldots, s$, the values $P_{i1}(\underline m),\ldots,P_{i\ell_i}(\underline m)$ are \underbar{not all} divisible by $a$.}

\vskip 1mm

\noindent
Then there exists $\underline m\in Z^k$ such that, for each $i=1,\ldots, s$, the values $P_{i1}(\underline m),\ldots, P_{i\ell_i}(\underline m)$
are
coprime in $Z$.
\end{definition}

This property is the one used by Schinzel (over $Z=\Zz$) in his 2002 paper \cite{schinzel2002}. The next definition introduces an alternate property, which is equivalent under some assumption, and which better suits our Hilbert-Schinzel context.

\begin{definition}
\label{th:primitive_hyp} 
Given an integral domain $Z$, we say that the {\it Primitive Specialization Hypothesis holds for $Z$} if for any integers $k,n, s \geq 1$ the following property is true:
\vskip 0,5mm

\noindent
\hskip 1pt ${\mathcal PriSpe}(k,n,s)$: Let $P_1(\underline t,\underline y), \ldots ,P_s(\underline t,\underline y)\in Z[\underline t,\underline y]$ be $s$ nonzero polynomials (in the variables $\underline t= (t_1,\ldots,t_k)$ and $\underline y=(y_1,\ldots,y_n)$). Assume that they are primitive \hbox{w.r.t.} $Q[\underline t]$ and that ${\mathcal F}_{\underline t}(P_1\cdots P_s)=\emptyset$. 
Then there exists $\underline m\in Z^k$ such that the polynomials $P_1(\underline m,\underline y), \ldots$, $P_s(\underline m,\underline y)$ are primitive \hbox{w.r.t.} $Z$.
\end{definition}

 \begin{lemma} \label{lem:reform}
{\rm  (1)} For all integers $k,s\geq 1$, we have: 
 \vskip 0,5mm
 
 \centerline{${\mathcal CopSch}(k,s) \Longrightarrow {\mathcal PriSpe}(k,n,s)$ for every $n\geq 1$.}
 \vskip 0,5mm
 
 \noindent
 {\rm  (2)} The converse holds as well

 {\rm (a)} if $Z$ has the property that every non-unit is divisible by a prime element, or, 
 
 {\rm (b)} in the special situation that $s=1$.
  \vskip 0,6mm

\noindent
{\rm  (3)} For all integers $k,n,s\geq 1$, we have: 
 \vskip 0,5mm
 
 \centerline{${\mathcal PriSpe}(k,n,1) \Longrightarrow {\mathcal PriSpe}(k,n,s)$.}
 \vskip 0,5mm

\end{lemma}

Thus the coprime Schinzel Hypothesis is stronger than the Primitive Specialization Hypothesis and they are equivalent in either case (a) or (b) from (2), in particular if $Z$ is a near UFD. Assertion (3) allows reducing to the case of one polynomial $P(\underline t, \underline y)\in Z[\underline t, \underline y]$ for proving the Primitive Specialization Hypothesis. For a near UFD, the analogous reduction to one family $\{P_{1}(\underline t),\ldots,P_{\ell}(\underline t)\}$ of polynomials in $Z[\underline t]$ also holds for the (equivalent) coprime Schinzel Hypothesis, but this is not clear for a general domain.

\begin{proof}
(1) Assuming ${\mathcal CopSch}(k,s)$ for some $k,s\geq 1$, let  $\underline y=(y_1,\ldots,y_n)$ be $n\geq 1$ variables and $P_i(\underline t,\underline y)$ ($i=1,\ldots,s$) as in ${\mathcal PriSpe}(k,n,s)$. Consider the sets $\{P_{i1}(\underline t),\ldots,P_{i\ell_i}(\underline t)\}$ ($i=1,\ldots,s$) of their respective coefficients in $Z[\underline t]$. Condition ${\mathcal F}_{\underline t}(P_1\cdots P_s)=\emptyset$
rewrites:
\vskip 1mm

\noindent
(*) \hskip 8mm
{\it $(\forall a\in Z\setminus Z^\times)$\hskip 1mm $(\exists \underline m\in Z^k)$ \hskip 1mm $(a$ {\rm does not divide} $\prod_{i=1}^sP_i(\underline m, \underline y))$}.

\vskip 1mm

\noindent
This obviously implies that

\vskip 1mm

\noindent
(**) \hskip 6mm
{\it $(\forall a\in Z\setminus Z^\times)$\hskip 1mm $(\exists \underline m\in Z^k)$ \hskip 1mm $(\forall i=1,\ldots,s)$\hskip 1mm $(a$ {\rm does not divide} $P_i(\underline m, \underline y))$},

\vskip 1mm

\noindent
which is equivalent to condition (AV)${}_{\underline t}$ for the sets $\{P_{i1}(\underline t),\ldots,P_{i\ell_i}(\underline t)\}$ ($i=1,\ldots,s$). If $\ell_1,\ldots,\ell_s$ are $\geq 2$, then the coprime Schinzel Hypothesis yields some $\underline m\in Z^k$ such that 
$P_{i1}(\underline m),\ldots, P_{i\ell_i}(\underline m)$ are 
coprime in $Z$, which equivalently translates as $P_i(\underline m,\underline y)$ being 
primitive \hbox{w.r.t.} $Z$ ($i=1,\ldots,s$). Taking Remark \ref{rem:technical_point}(1) below into account, we obtain that ${\mathcal CopSch}(k,s)$ implies ${\mathcal PriSpe}(k,n,s)$.

\begin{remark} \label{rem:technical_point} (1) If $\ell_i=1$ for some $i\in\{1,\ldots,s\}$, \hbox{i.e.} if the polynomial $P_i(\underline t,\underline y)$ is a monomial in $\underline y$, then this polynomial should be treated independently. In this case, $P_i(\underline t,\underline y)$ being primitive \hbox{w.r.t.} $Z[\underline t]$, it is of the form $c y_1^{i_1}\cdots y_n^{i_n}$ for some integers $i_1,\ldots,i_n \geq 0$ and $c\in Z^{\times}$. Then $P_i(\underline m,\underline y) = c y_1^{i_1}\cdots y_n^{i_n}$ remains primitive \hbox{w.r.t.} $Z$ {\it for every $\underline m\in Z^k$}.
\end{remark}

\noindent
(2) Suppose given sets $\{P_{i1}(\underline t),\ldots,P_{i\ell_i}(\underline t)\}$ as in ${\mathcal CopSch}(k,s)$ for some integers $k,s\geq 1$. Consider
the polynomials 
\vskip 0,5mm

\centerline{$P_i(\underline t,\underline y)= P_{i1}(\underline t)y_1+ \cdots + P_{i\ell_i}(\underline t)y_{\ell_i}$, $i=1,\ldots,s$,}
\vskip 0,5mm

\noindent
where $\underline y=(y_1,\ldots,y_\ell)$ is an $\ell$-tuple of new variables and $\ell = \max(\ell_1,\ldots,\ell_s)$\footnote{Any polynomial in $Z[\underline t][y_1,\ldots,y_{\ell_i}]$ with coefficients $P_{i1}(\underline t),\ldots,P_{i\ell_i}(\underline t)$ may be used instead of the polynomial $P_{i1}(\underline t)y_1+ \cdots + P_{i\ell_i}(\underline t)y_{\ell_i}$ ($i=1,\ldots,s$).}. Condition (AV)${}_{\underline t}$ rewrites as (**) above.
In either one of the situations (a) or (b) of the statement, condition (**) does imply condition (*), and equivalently ${\mathcal F}_{\underline t}(P_1\cdots P_s)=\emptyset$, for the polynomials 
$P_i(\underline t,\underline y)$ defined above. Thus if ${\mathcal PriSpe}(k,n,s)$ holds (for every $n\geq 1$), we obtain some $\underline m\in Z^k$ such that the polynomials $P_1(\underline m,\underline y), \ldots$, $P_s(\underline m,\underline y)$ are 
primitive \hbox{w.r.t.} $Z$, which, in terms of the original polynomials $P_{ij}(\underline t)$, corresponds to the conclusion of the requested property ${\mathcal CopSch}(k,s)$.
\vskip 1mm

\noindent
(3) Let $P_1(\underline t,\underline y),\ldots, P_s(\underline t,\underline y)$ be as in ${\mathcal PriSpe}(k,n,s)$. Set $P=P_1\cdots P_s$. We have $P\in Z[\underline t,\underline y]$, and from Gauss's lemma (applied to the UFD $Q[\underline t]$), $P$ is primitive \hbox{w.r.t.} $Q[\underline t]$. By hypothesis, ${\mathcal F}_{\underline t}(P)=\emptyset$. Assuming ${\mathcal PriSpe}(k,n,1)$, it follows that there  exists $\underline m\in Z^k$ such that $P(\underline m,\underline y)$ is 
primitive \hbox{w.r.t.} $Z$. But then, the polynomials $P_1(\underline m,\underline y), \ldots, P_s(\underline m,\underline y)$ are 
primitive \hbox{w.r.t.} $Z$ as well.
\end{proof}

\subsection{Main result} \label{ssec:main-conclusions} Before stating the main conclusions on our Hypotheses, we introduce
the following variant of the Primitive Specialization Hypothesis, which is more precise and more involved but turns out to be 
quite central and useful.

\begin{definition} \label{def:strongPSH}
Let $Z$ be an integral domain such that every nonzero $a\in Z$ has only finitely many prime divisors (modulo units). We say that the {\it \underbar{starred} Primitive Specialization Hypothesis holds for $Z$} if the following property is true for any integers $k,s \geq 1$:

\vskip 1,5mm

\noindent
${\mathcal PriSpe}^{\displaystyle \star}(k,s)$: 
{\it Let $n\geq 1$ and let $P_1(\underline t,\underline y), \ldots ,P_s(\underline t,\underline y)\in Z[\underline t,\underline y]$ be $s$ nonzero polynomials (with $\underline t=(t_1,\ldots,t_k), \underline y=(y_1,\ldots,y_n)$), primitive \hbox{w.r.t.} $Q[\underline t]$  and such that  the product $P_1\cdots P_s$ has no fixed \underbar{prime} divisor
in $Z$ \hbox{w.r.t.} $\underline t$.
Let $P_0\in Z[\underline t]$, \hbox{$P_0\not=0$.}
 Then there exists  $(m_1, \ldots, m_{k-1}) \in Z^{k-1}$ and an arithmetic progression $\tau_k = (\omega_k \ell + \alpha_k)_{\ell\in Z}$ with $\omega_k,\alpha_k\in Z$, $\omega_k\not=0$, such that for all but finitely many 
 $m_k\in \tau_k$, the polynomials
\vskip 1mm

\noindent
 \hskip 25mm $P_1(m_1,\ldots,m_{k},\underline y), \ldots, P_s(m_1,\ldots,m_{k},\underline y)$
\vskip 1mm

\noindent
have no \underbar{prime} divisors in $Z$,
and $P_0(m_1,\ldots,m_{k})\not=0$. }

\noindent
({\it with the convention that for $k=1$, existence of $(m_1, \ldots, m_{k-1}) \in Z^{k-1}$ is not requested}).
\end{definition}

If $Z$ is a near UFD (and so it is equivalent to request that prime or non-unit divisors exist), ${\mathcal PriSpe}^{\displaystyle \star}(k,s)$ is a more precise form of (${\mathcal PriSpe}(k,n,s)$ {\it for all} $n\geq 1$) saying where to find the tuples $\underline m$, the existence of which is asserted in ${\mathcal PriSpe}(k,n,s)$; thus we have 
\vskip 1mm

\centerline{${\mathcal PriSpe}^{\displaystyle \star}(k,s) \Rightarrow ({\mathcal PriSpe}(k,n,s)$ {\it for all} $n\geq 1$) $ \Leftrightarrow {\mathcal CopSch}(k,s)$.}
\vskip 1mm

\noindent
Because the existence of non unit \hbox{vs.}  prime divisors issue is not void in Dedekind domains, Definition \ref{th:primitive_hyp} and Definition \ref{def:strongPSH} do not compare so obviously for Dedekind domains. 

\begin{theorem} \label{thm:main_theorem_gen} {\rm (a)} If $Z$ is a near UFD or a Dedekind domain, then the starred Primitive Specialization Hypothesis holds for $Z$.
\vskip 0,5mm

\noindent
{\rm (b)} If $Z$ is a near UFD, then both the Coprime Schinzel Hypothesis and the Primitive Specialization Hypotheses
 hold.
\vskip 0,5mm

\noindent
{\rm (c)} If $Z$ is a Dedekind domain, then the Primitive Specialization Hypothesis holds with $k=1$ and $n,s\geq 1$ (\hbox{i.e.} for polynomials with one variable $t$ to be specialized), and the Coprime Schinzel Hypothesis holds if in addition $s=1$ (one polynomial $P$).
\end{theorem}

Theorem \ref{thm:main_theorem} for UFDs is the special case $s=k=1$ of Theorem \ref{thm:main_theorem_gen}(b), and Theorem \ref{thm:main_theorem} for Dedekind domains 
(already proved in  \S \ref{ss:ded}) is the second part of Theorem \ref{thm:main_theorem_gen}(c).

\subsection{Two lemmas} \label{ssec:chinese}  
The following lemmas are used in the proof of Theorem \ref{thm:main_theorem_gen} and of Theorem \ref{thm-intro}.  The first one is a refinement of the Chinese Remainder Theorem.

\begin{lemma} \label{lemma:3} Let $Z$ be an integral domain. Let $I_1,\ldots, I_\rho$ be $\rho$ maximal ideals of $Z$ and $I_{\rho+1},\ldots,I_r$ be $r-\rho$ ideals assumed to be prime but not maximal (with $0\leq \rho \leq r$). Assume further that $I_j \not\subset I_{j^\prime}$ for any distinct elements $j,j^\prime \in \{1,\ldots,r\}$. Let $\underline t$, $\underline y$ be tuples of variables of length $k, n\geq 1$ and let $F\in Z[\underline t,\underline y]$ be a polynomial, nonzero modulo each ideal $I_j$, $j=\rho+1,\ldots,r$. Then for every $(\underline a_1,\ldots,\underline a_\rho)\in (Z^k)^\rho$, there exists $\underline m\in Z^k$ such that $\underline m \equiv \underline a_j \pmod{I_j}$ for each $j=1,\ldots,\rho$, and $F(\underline m,\underline y) \not\equiv 0 \pmod{I_j}$ for each $j=\rho+1,\ldots,r$.
\end{lemma}

\begin{proof}  
Assume first $0\leq \rho<r$. A first step is to show by induction on $r-\rho\geq 1$ that there exists $\underline m_0\in Z^k$ such that $F(\underline m_0,\underline y) \not\equiv 0 \pmod{I_j}$, $j=\rho+1,\ldots,r$. 

Start with $r-\rho =1$. The quotient ring $Z/I_{\rho+1}$ is an integral domain but not a field, hence is infinite; and $F$ is nonzero modulo 
$I_{\rho+1}$ (\hbox{i.e.} in $(Z/I_{\rho+1})[\underline t,\underline y]$). Thus elements $\underline m_0\in Z^k$ exist such that $F(\underline m_0,\underline y) \not\equiv 0 \pmod{I_{\rho+1}}$. Assume next that there is an element of $Z^k$, say $\underline m_1$, such that $F(\underline m_1,\underline y) \not\equiv 0 \pmod{I_j}$, $j=\rho+1,\ldots,s$ with $s<r$. It follows from the assumptions
on $I_{\rho+1},\ldots,I_r$ that the product $I_{\rho+1} \cdots I_s$ is not contained in $I_{s+1}$. Pick an element $\pi$ in 
$I_{\rho+1} \cdots I_s$ that is not in $I_{s+1}$ and consider the polynomial $F(\underline m_1 + \pi \underline t,\underline y)$. This polynomial is nonzero modulo $I_{s+1}$ since both $F$ and $\underline m_1 + \pi \underline t$ are nonzero modulo $I_{s+1}$. As above, the quotient ring $Z/I_{s+1}$ is an infinite integral domain and so there exists $\underline t_0\in Z^k$ such that $F(\underline m_1+\pi \underline t_0,\underline y) \not\equiv 0 \pmod{I_{s+1}}$; and for each $j=\rho+1,\ldots,s$, since $\pi\in I_j$,  we have $F(\underline m_1+\pi \underline t_0,\underline y) \equiv F(\underline m_1,\underline y)\not\equiv 0 \pmod{I_{j}}$.  
Set $\underline m_0 = \underline m_1+\pi \underline t_0$ to conclude the induction.

Set $J = I_{\rho+1}\cdots I_r$. The ideals $I_1,\ldots, I_\rho, J$ are pairwise comaximal.
We may apply the Chinese Remainder Theorem, and will, component by component. More specifically write $\underline m_0=(m_{01},\ldots,m_{0k})$ and $\underline a_i=(a_{i1},\ldots,a_{ik})$, $i=1,\ldots,\rho$. For each $h=1,\ldots,k$, there is an element $m_h\in Z$ such that 
$m_h \equiv a_{jh} \pmod{I_j}$ for each $j=1,\ldots,\rho$, and $m _h\equiv m_{0h} \pmod{J}$. Set $\underline m=(m_1,\ldots,m_k)$. Clearly we have $\underline m \equiv \underline a_j \pmod{I_j}$ for each $j=1,\ldots,\rho$, and $\underline m \equiv \underline m_{0}  \pmod{J}$. The last congruence implies that, for each $j=\rho+1,\ldots,r$, we have $\underline m \equiv \underline m_{0}  \pmod{I_j}$, and so $F(\underline m,\underline y)\not\equiv 0 \pmod{I_j}$.

If $r=\rho$, then there is no ideal $J$ and the sole second part of the argument, applied with the maximal ideals $I_1,\ldots, I_\rho$, yields the result.
\end{proof}

\begin{lemma} \label{lem:cond-star} Let $Z$ be an integral domain such that every nonzero element $a\in Z$ has only finitely many prime divisors (modulo units). Then, for every real number $B>0$, there are only finitely many prime principal ideals $pZ$ 
of norm $|Z/pZ|$ less than or equal to $B$.
\end{lemma}

The assumption on $Z$ is satisfied in particular if $Z$ is a near UFD or a Dedekind domain.

\begin{proof} One may assume that $Z$ is infinite. Fix a real number $B>0$. For every prime power $q=\ell^r \leq B$, 
pick an element $m_q\in Z$ such that $m_q^q - m_q \not=0$. Let $a$ be the product of all elements $m_q^q - m_q$ with $q$ running over all prime powers $q \leq B$. From the assumption on $Z$, the list, say ${\mathcal D}_a$, of all prime divisors of $a$ (modulo units), is finite.

Consider now a prime $p\in Z$ such that $|Z/pZ| \leq B$. The integral domain $Z/pZ$, being finite, is a field. Hence  $|Z/pZ|$ is a prime power $q= \ell^r$ ; and $q\leq B$. Of course we have $m_q^q - m_q \equiv 0 \pmod{p}$. Hence $p$ divides $a$, \hbox{i.e.} $p\in {\mathcal D}_a$. \end{proof}

\subsection{A reduction lemma} \label{ssec:reduction} 
This lemma is a central tool of the proof of Theorem \ref{thm:main_theorem_gen}.

\begin{lemma}\label{lem:reduction1} Let $Z$ be an integral domain such that every nonzero $a\in Z$ has only finitely many prime divisors (modulo units).
Let $k,s\geq 1$ be two integers. Assume that property ${\mathcal PriSpe}^{\displaystyle \star}(1,1)$ holds.
Then property ${\mathcal PriSpe}^{\displaystyle \star}(k,s)$ holds for all integers $k,s\geq 1$.
\end{lemma} 

This reduction is explained in the proof of \cite[Theorem 1]{schinzel2002} (see pages 242--243), with two differences. First, Schinzel uses the coprime Hypothesis formulation instead of the Primitive Specialization one (from Lemma \ref{lem:reform}, they are equivalent if $Z$ is a near UFD). Secondly, Schinzel works 
over $Z=\Zz$. Our proof adapts his arguments to the Primitive Specialization formulation and shows that they carry over to 
our more general domains $Z$.

\begin{proof}[Proof of Lemma \ref{lem:reduction1}]
Observe that using the same argument as for Lemma \ref{lem:reform}(3), one can reduce from $s$ to one polynomial $P(\underline t, \underline y)$: ${\mathcal PriSpe}^{\displaystyle \star}(k,1) \Leftrightarrow {\mathcal PriSpe}^{\displaystyle \star}(k,s)$. Thus we merely need to prove that if ${\mathcal PriSpe}^{\displaystyle \star}(1,1)$ holds then so does ${\mathcal PriSpe}^{\displaystyle \star}(k,1)$ ($k\geq 1$).

Assume that ${\mathcal PriSpe}^{\displaystyle \star}(1,1)$ holds. By the induction principle, we need to prove that
\vskip 0,5mm

\centerline{${\mathcal PriSpe}^{\displaystyle \star}(k-1,1) \Rightarrow {\mathcal PriSpe}^{\displaystyle \star}(k,1)$ for $k\geq 2$.}
\vskip 0,5mm

\noindent
Assume ${\mathcal PriSpe}^{\displaystyle \star}(k-1,1)$ and  let $n$, 
$P(\underline t,\underline y)$ and $P_0(\underline t)$ be as in  ${\mathcal PriSpe}^{\displaystyle \star}(k,1)$.

One may assume that $\deg_{t_k}(P) >0$. 
Let ${\mathcal P}$ be the set of all primes $p\in Z$ (modulo units) such that $|Z/pZ| \leq \max_{1\leq h \leq k} \deg_{t_h}(P)$. 
By Lemma \ref{lem:cond-star}, the set ${\mathcal P}$ is finite. Let $\pi$ be the product of its elements.

The first step is to construct a $k$-tuple $\underline u= (u_{1},\ldots,u_{k}) \in Z^k$ such that 
\vskip 1mm

\noindent
(1)  \hskip 30mm $P(\underline u,\underline y) \not\equiv 0 \pmod{p}$ for every $p\in {\mathcal P}$.

\vskip 1mm

By assumption, no prime of $Z$ is a fixed divisor of $P$ \hbox{w.r.t.} $\underline t$. Thus for every $p\in {\mathcal P}$,  there exists a $k$-tuple $\underline u_p = (u_{p1}, \ldots, u_{pk}) \in Z^k$ such that $P(\underline u_p,\underline y)\not\equiv 0 \pmod{p}$.
Denote by ${\mathcal P}_1 \subset {\mathcal P}$ the subset of primes $p$ such that the ideal $pZ$ is maximal in $Z$.
We now apply Lemma \ref{lemma:3}. From above, $P$ is nonzero modulo each $p\in {\mathcal P}\setminus {\mathcal P}_1$, and we have $pZ \not\subset p^\prime Z$ for any distinct 
$p,p^\prime \in {\mathcal P}$. Thus
Lemma \ref{lemma:3} provides a  $k$-tuple $\underline u_0= (u_{01},\ldots,u_{0k}) \in Z^k$ such that
 $\underline u_0 \equiv \underline u_p \pmod{p}$ for every prime $p\in{ \mathcal P}_1$ and 
$P(\underline u_0,\underline y)  \not\equiv 0 \pmod{p}$ for every $p\in {\mathcal P}\setminus {\mathcal P}_1$. These congruences imply that $P(\underline u_0,\underline y) \not\equiv 0 \pmod{p}$ for every $p\in {\mathcal P}$.
Furthermore, denoting by $\tau_h$ the arithmetic progression $\tau_h= (\pi \ell + u_{0h})_{\ell\in Z}$ ($h=1,\ldots,k$), conclusion (1) holds for every $\underline u = (u_1,\ldots,u_k)\in \tau_1 \times\cdots \times \tau_k$. Fix such a $k$-tuple $\underline u$.

Consider the following polynomial, where $\underline v^\prime = (v_1,\ldots,v_{k-1})$ is a tuple of new variables, and ${\underline u}^\prime = (u_1,\ldots, u_{k-1})$:

\vskip 1mm

\centerline{$\widetilde P (\underline v^\prime, t_k,\underline y)= P(\pi \hskip 1pt \underline v^\prime+{\underline u}^\prime, t_k, \underline y) \in Z[\underline v^\prime,t_k,\underline y]$.}
\vskip 1mm

\noindent
We check below that as a polynomial in the $n+1$ variables $t_k,\underline y$ with coefficients in $Z[\underline v^\prime]$, it satisfies the assumptions allowing using the induction hypothesis ${\mathcal PriSpe}^{\displaystyle \star}(k-1,1)$.

Clearly $\widetilde P$ is nonzero. Set $\underline t^\prime = (t_1,\ldots,t_{k-1})$. 
The polynomial $P$ is primitive \hbox{w.r.t.} $Q[\underline t^\prime]$. Hence $\widetilde P$ is primitive \hbox{w.r.t.} $Q[\underline v^\prime]$. 
Also note that by (1), we have, for every $p\in {\mathcal P}$,
\vskip 1,5mm

\noindent
(2) \hskip 20mm $\widetilde P(\underline \ell^\prime, u_k,\underline y) \not\equiv 0 \pmod{p}$  \hskip 2mm for every $\underline \ell^\prime \in Z^{k-1}$. 
\vskip 1,5mm

\noindent
The next paragraph shows that $\widetilde P$ has no fixed prime divisor in $Z$ \hbox{w.r.t.} $\underline v^\prime$.

Assume that  
there is a prime $p\in Z$ such that $\widetilde P(\underline \ell^\prime,t_k,\underline y) \equiv 0 \pmod{p}$ for every $\underline \ell^\prime \in Z^{k-1}$. It follows from (2) that $p\notin {\mathcal P}$. This gives that for every $h=1,\ldots,k-1$, we have $\pi v_h + u_h \not\equiv 0 \pmod{p}$ (as a polynomial in $v_h$). 
This, conjoined with 
$P(\underline t^\prime,t_k,\underline y)=P(\underline t,\underline y) \not\equiv 0 \pmod{p}$ 
shows that $P(\pi \underline v^\prime + {\underline u}^\prime,t_k,\underline y)\not\equiv 0 \pmod{p}$. In other words, $\widetilde P$
is nonzero modulo $p$.
 A contradiction then follows from Lemma \ref{lem:prelim-proof-R[u]}(a) and 
\vskip 0,5mm

\centerline{$\displaystyle |Z/pZ|>\max_{1\leq h\leq k-1}(\deg_{t_h}(P)) = \max_{1\leq h\leq k-1}(\deg_{v_h}(\widetilde P))$.}

\vskip 0,5mm

We will apply assumption ${\mathcal PriSpe}^{\displaystyle \star}(k-1,1)$ to $\widetilde P\in Z[\underline v^\prime] [t_k,\underline y]$, and for the following choice of a nonzero polynomial $\widetilde P_0 \in Z[\underline v^\prime]$. The polynomial $P$ is primitive \hbox{w.r.t.} $Q[\underline t]$. Thus if $\{P_{j}(\underline t) \hskip 1pt | \hskip 1pt j\in J\}$ is the set of coefficients of $P$ (viewed as a polynomial in $\underline y$), by writing a {B\'ezout} relation in the PID $Q(\underline t^\prime)[t_k]$ and then clearing the denominators, we obtain elements $A_{j}\in Z[\underline t]$ and $\Delta\in Z[\underline t^\prime]$, $\Delta\not=0$, such that 
\vskip 0,5mm

\noindent
(3) \hskip 30mm $\displaystyle \sum_{j\in J} A_{j}(\underline t^\prime,t_k) P_{j}(\underline t^\prime,t_k) = \Delta(\underline t^\prime)$. 
\vskip 0,5mm

\noindent
Set then 
\vskip 0,5mm

\centerline{$\displaystyle  \widetilde P_0 (\underline v^\prime) = P_{0 \infty}(\pi \underline v^\prime + \underline u^\prime) \times  \Delta(\pi \underline v^\prime + {\underline u}^\prime)$}

\vskip 0,5mm

\noindent
where $P_{0 \infty}(\underline t^\prime)\in Z[\underline t^\prime]$ is the leading coefficient of $P_0$ viewed as a polynomial in $t_k$.

From ${\mathcal PriSpe}^{\displaystyle \star}(k-1,1)$, there exists 
$\underline \ell^\prime=(\ell_1,\ldots,\ell_{k-1})\in Z^{k-1}$ such that  
\vskip 1mm

\noindent
{\rm (4)} the polynomial $\widetilde P(\underline \ell^\prime,t_k,\underline y) = P(\pi \underline \ell^\prime +{\underline u}^\prime,t_k, \underline y)$  \hbox{has no prime divisors in $Z$,}
\vskip 1mm

\noindent
and
\vskip 1mm

\noindent
{\rm (5)} \hskip 20mm $\widetilde P_0(\underline \ell^\prime) = P_{0 \infty}(\pi \underline \ell^\prime + \underline u^\prime) \times\Delta(\pi \underline \ell^\prime+ {\underline u}^\prime) \not= 0$.
\vskip 1mm

It follows from (3) and (5) that the polynomial $\widetilde P(\underline \ell^\prime,t_k,\underline y)$ is nonzero and primitive \hbox{w.r.t.} $Q[t_k]$. We check below that $\widetilde P(\underline \ell^\prime, t_k, \underline y)$ has no fixed prime divisor 
in $Z$ \hbox{w.r.t.} the variable $t_k$.

Assume that for some prime $p\in Z$, we have $\widetilde P(\underline \ell^\prime, m_k, \underline y) \equiv 0 \pmod{p}$ for every $m_k \in Z$. 
In view of (1), we have $p\notin {\mathcal P}$. Hence, by choice of ${\mathcal P}$, we have 
\vskip 1mm

\noindent
(6) \hskip 20mm $|Z/pZ| > \deg_{t_k}(P) = \deg_{t_k}(\widetilde P) \geq \deg_{t_k} \left(\widetilde P(\underline \ell^\prime, t_k, \underline y) \right)$
\vskip 1mm

\noindent
By (4), the polynomial $\widetilde P(\underline \ell^\prime,t_k,\underline y)$ is nonzero modulo $p$ (\hbox{i.e.} nonzero in $(Z/pZ)[t_k,\underline y]$). This, conjoined with (6), contradicts Lemma \ref{lem:prelim-proof-R[u]}(a).

Use next assumption ${\mathcal PriSpe}^{\displaystyle \star}(1,1)$  to conclude that 
there exists an arithmetic progression $\tau_k=(\omega_k \ell+\alpha_k)_{\ell \in Z}$, with $\omega_k, \alpha_k \in Z$, $\omega_k\not=0$, such that for every $m_k\in \tau_k$, the 
polynomial 
\vskip 1mm

\noindent
\hskip 30mm $\widetilde P(\underline \ell^\prime, m_k, \underline y)= P(\pi \underline \ell^\prime + {\underline u}^\prime,m_k, \underline y)$
\vskip 1mm

\noindent
has no prime divisors in $Z$.
Furthermore, taking into account that $P_{0 \infty}(\pi \underline \ell^\prime + \underline u^\prime) \not=0$ (by (5)), we have $P_0(\pi \underline \ell^\prime + {\underline u}^\prime,m_k)\not=0$ for all finitely many $m_k\in \tau_k$. The requested conclusion is thus proved for $(m_1,\ldots,m_{k-1}) = \pi \underline \ell^\prime + {\underline u}^\prime$ and the arithmetic progression $\tau_k$.
\end{proof}

\subsection{Proof of Theorem \ref{thm:main_theorem_gen}}  \label{ssec:proof_k=1} 

\subsubsection{Proof of (a)} \label{pf-of-(a)}
Let $Z$ be a near UFD or a Dedekind domain. By Lemma \ref{lem:reduction1}, proving  ${\mathcal PriSpe}^{\displaystyle \star}(1,1)$ will give 
${\mathcal PriSpe}^{\displaystyle \star}(k,s)$ for all integers $k,s\geq 1$, \hbox{i.e.} the starred Primitive Specialization Hypothesis.
\vskip 0,5mm

\noindent
{\it 1st case:} {\it $Z$ is a near UFD.} 
Let $P(t,\underline y)$ be as in ${\mathcal PriSpe}^{\displaystyle \star}(1,1)$. From Remark \ref{rem:technical_point}(1), one may assume that $P$ is not a monomial in $\underline y$. Let $\delta\in Z$ be the parameter associated in \S \ref{ssec:preliminaries} with the nonzero coefficients $P_{1}(t),\ldots,P_{\ell}(t) \in Z[t]$ of $P(t,\underline y)$ viewed as a polynomial in $\underline y$. 

Let $p_1,\dots, p_r$ be the prime divisors  of $\delta$ (distinct modulo units).  
From condition ${\mathcal F}_{t}(P)=\emptyset$, for every $h=1,\ldots,r$, there exists $m_h\in Z$ such that $P(m_h,\underline y)\not\equiv 0 \pmod{p_h}$.

We may assume without loss of generality that, for some $\rho\in \{0,1,\ldots,r\}$, the ideals $p_1Z, \dots, p_\rho Z$ 
are maximal in $Z$, whereas the ideals $p_{\rho+1}Z, \dots, p_{r}Z$ are not. 
From above, $P$ is nonzero modulo each $p_h$, $h=\rho+1,\ldots,r$, and we have $p_iZ \not\subset p_{i^\prime} Z$ for any distinct 
$i,i^\prime \in  \{1,\ldots,r\}$. 
Lemma \ref{lemma:3} (with $\underline t=t$), applied with $I_h = p_hZ$, $h=1,\ldots,r$, yields that there exists an element $m_0\in Z$ such that $m_0 \equiv m_h \pmod{p_h}$, $h=1,\ldots,\rho$, and $P(m_0,\underline y) \not\equiv 0 \pmod{p_h}$, $h=\rho+1,\ldots,r$. These congruences imply that $P(m_0,\underline y) \not\equiv 0 \pmod{p_h}$, for each $h=1,\ldots,r$.

Conclude that the polynomial $P(m_0,\underline y)$ is primitive \hbox{w.r.t.} $Z$. Indeed assume that some non-unit $a\in Z$ divides $P(m_0,\underline y)$. From Lemma \ref{lemma:1}, $a$ then divides $\delta$. But using that $Z$ is a near UFD, we obtain that some prime divisor of $a$ divides $\delta$ and $P(m_0,\underline y)$, contrary to what we have established.  
Furthermore, our conclusion ``{\it $P(m,\underline y)$ is primitive \hbox{w.r.t.} $Z$}'' is not only true for $m=m_0$, but also for every $m$ in the arithmetic progression $(\delta \ell + m_0)_{\ell\in Z}$.

\begin{remark}\label{rem:plus2} 
 More generally, if several polynomials $P_i(t,\underline y)$ are given as in ${\mathcal PriSpe}(1,n,s)$, then, denoting the corresponding $\delta$-parameters by $\delta_1,\ldots,\delta_s$, the conclusion of the Primitive Specialization Hypothesis ``{\it $P_1(m,\underline y),\ldots,P_s(m,\underline y)$ are primitive \hbox{w.r.t.} $Z$}'' holds for every $m$ in some arithmetic progression $(\omega \ell + \alpha)_{\ell\in Z}$. Specifically one can take $\omega,\alpha\in Z$ such that every prime divisor $p$ of $\delta= \delta_1\cdots \delta_s$ divides $\omega$ and satisfies $\prod_{i=1}^s P_i(\alpha ,\underline y) \not\equiv 0 \pmod{p}$. 
\end{remark}

\vskip 0,5mm

\noindent
{\it 2nd case:} {\it $Z$ is a Dedekind domain.} 
We noted in Proposition \ref{rem:plus} that the proof from \S \ref{ss:ded} of the  Dedekind domain part of Theorem \ref{thm:main_theorem} gives 
a more precise form of ${\mathcal CopSch}(1,1)$. Namely it produces a whole arithmetic progression, the elements of which satisfy
the conclusion of the ${\mathcal CopSch}(1,1)$ (Proposition \ref{rem:plus}(a)). Furthermore, Proposition \ref{rem:plus}(b) shows
that the same holds if non-unit divisors are replaced by prime divisors in the statement of the property. Finally, adjusting the 
observation made in Lemma \ref{lem:reform}(1) to show that ``Coprime Schinzel''  is stronger than ``Primitive Specialization'', what 
we eventually obtain from this chain of arguments is that the proof from \S \ref{ss:ded} indeed gives 
property ${\mathcal PriSpe}^{\displaystyle \star}(1,1)$.

\subsubsection{Proof of (b)} This clearly follows from (a) as for near UFDs, the {\it starred} Primitive Specialization Hypothesis is stronger than the Primitive Specialization Hypothesis,
and that by Lemma \ref{lem:reform}(2-a), the latter and the coprime Schinzel Hypothesis are equivalent.

\subsubsection{Proof of (c)} Let $Z$ be a Dedekind domain.
As already pointed out, the second part of (c), \hbox{i.e.} ${\mathcal CopSch}(1,1)$, is the Dedekind domain part of Theorem \ref{thm:main_theorem} proved in  \S \ref{ss:ded}. By Lemma \ref{lem:reform}(1), we deduce ${\mathcal PriSpe}(1,n,1)$ for every $n\geq 1$. Combined with Lemma \ref{lem:reform}(3), we obtain ${\mathcal PriSpe}(1,n,s)$ for all integers $n,s\geq 1$. $\qed$

\section{The Hilbert-Schinzel specialization property} \label{ssec:SH_more_general_rings} 
The goal of this section is to show Theorem \ref{thm-intro}. We distinguish two cases: $k=1$ in \S \ref{ssec:final_k=1} and $k\geq 1$ in \S \ref{ssec:final_k bigger}. We consider refinements of Theorem \ref{thm-intro} in \S \ref{ssec:refine}.
\vskip 1mm

A new assumption on $Z$ in this section is that it is a {\it Hilbertian ring}. 

\begin{definition} \label{def:HilbertUFD} Let $Z$ be an integral domain such that the fraction field $Q$ is of characteristic $0$ or imperfect. The ring $Z$ is  a {\it Hilbertian ring} if the following holds: 

\noindent
Let $\underline t=(t_1,\ldots,t_k)$, $\underline y = (y_1,\ldots, y_n)$ be tuples of variables $(k, n\geq 1)$, $P_1(\underline t,\underline y),\ldots, P_s(\underline t,\underline y)$ be $s$ polynomials ($s\geq 1$), irreducible in $Q[\underline t,\underline y]$, of degree at least $1$ in $\underline y$, and  $F(\underline t) \in Q[\underline t]$ be a nonzero polynomial. Then the so-called Hilbert subset 
\vskip 1mm

\centerline{$H_Q(P_1,\ldots,P_s;F)= \left\{\underline m\in Q^k \hskip 1pt \left |\hskip 2pt \begin{matrix} P_i(\underline m,\underline y)\hskip 1mm  \hbox{is irreducible in $Q[\underline y]$}\hfill \\
\hfill (i=1,\ldots,s) \hfill\\
\hbox{and} \hskip 1mm F(\underline m)\not= 0. \hfill \\
\end{matrix}\right. \right\}$}

\vskip 1mm

\noindent
contains a $k$-tuple 
 $\underline m \in Z^k$. 
If $Z$ is a field, \hbox{it is called a {\it Hilbertian field}.} \end{definition}

The original definition of Hilbertian ring from \cite[\S 13.4]{FrJa} has the defining condition only requested for $n=1$ and polynomials $P_1(\underline t,y_1),\ldots P_s(\underline t,y_1)$ that are further assumed to be separable in $y_1$. But
\cite[Proposition 4.2]{BDN19} shows that it is equivalent to Definition \ref{def:HilbertUFD} under the imperfectness condition. Note further that since Zariski open subsets of Hilbert subsets remain Hilbert subsets, it is equivalent to require in Definition \ref{def:HilbertUFD} that $H_Q(P_1,\ldots,P_s;F)$ contains a Zariski-dense subset of tuples $\underline m\in Z^k$. In particular, a Hilbertian ring $Z$ is necessarily infinite.
\vskip 1mm

For the proof of Theorem \ref{thm-intro}, note that one may assume that 
none of the polynomials $P_1,\ldots,P_s$ are monomials in $\underline y$ (with coefficients in $Z[\underline t]$): to fulfill the assumptions of Theorem \ref{thm-intro}, such a monomial must be of the form $c y_i$ for some $c\in Z^\times$ and $i\in \{1,\ldots,n\}$; the required conclusion for this monomial then trivially holds for every $\underline m\in Z^k$.

\subsection{Proof of Theorem \ref{thm-intro} - case $k=1$} \label{ssec:final_k=1} We will prove more generally that the Hilbert-Schinzel property holds with $k=1$ and 
given integers $n,s\geq 1$ if $Z$ is a Hilbertian ring and the Primitive Specialization Hypothesis holds for $Z$ with $k=1$ and the integers $n$ and $s$. From Theorem \ref{thm:main_theorem_gen}, the latter holds if $Z$ is a near UFD or a Dedekind domain.

Let $P_1(t,\underline y), \ldots, P_s(t,\underline y)$ be $s$ \hbox{polynomials}, irreducible in $Q[t,\underline y]$, primitive \hbox{w.r.t.} $Z$, of degree $\geq 1$ in $\underline y$ and such that ${\mathcal F}_{t}(P_1\cdots P_s) = \emptyset$. The polynomials $P_1(t,\underline y), \ldots, P_s(t,\underline y)$ are also primitive \hbox{w.r.t.} $Q[t]$, as they are irreducible in $Q[t,\underline y]$.
The Primitive Specialization Hypothesis with $k=1$ provides an element $\alpha \in Z$ such that $P_1(\alpha,\underline y), \ldots, P_s(\alpha,\underline y)$ are primitive \hbox{w.r.t.} $Z$.
Let $\delta_i$ be the parameter associated  in Lemma \ref{lemma:1} with the coefficients $P_{ij}(t) \in Z[t]$ of  $P_i$ ($i=1,\ldots,s$), and let $\omega$ be a multiple of $\delta=\delta_1\cdots \delta_ s$. From Lemma \ref{lemma:1}, for every $\ell \in Z$, the polynomials 
\vskip 1mm

\centerline{$P_{1}(\alpha+ \ell \omega,\underline y),\ldots,P_{s}(\alpha+ \ell \omega,\underline y)$}
\vskip 1mm

\noindent
are primitive \hbox{w.r.t.} $Z$. Consider next the polynomials 
\vskip 1mm

\centerline{$P_1(\alpha+\omega t,\underline y), \ldots, P_s(\alpha+\omega t,\underline y)$.}
\vskip 1mm

\noindent
They are in $Z[t,\underline y]$ and are irreducible in $Q[t,\underline y]$. As $Z$ is a Hilbertian ring, infinitely many $\ell \in Z$ exist such that the polynomials
\vskip 1mm

\centerline{
$P_1(\alpha+\omega \ell,\underline y), \ldots, P_s(\alpha+\omega \ell,\underline y)$}
\vskip 1mm

\noindent 
are irreducible in $Q[\underline y]$. From above, these polynomials are also primitive \hbox{w.r.t.} $Z$.
\vskip 1mm

\subsection{Proof of Theorem \ref{thm-intro} - case $k\geq 1$}  \label{ssec:final_k bigger}
 We may assume that we are in case (a), and $Z$ is a near UFD. Let $P_1(\underline t,\underline y), \ldots, P_s(\underline t,\underline y)$ be $s$ \hbox{polynomials}, irreducible in $Q[\underline t,\underline y]$, primitive \hbox{w.r.t.} $Z$, of degree $\geq 1$ in $\underline y$ and such that ${\mathcal F}_{\underline t}(P_1\cdots P_s) = \emptyset$. 
Fix a nonzero polynomial $P_0\in Z[\underline t]$. We need to produce a $k$-tuple $\underline m\in Z^k$ such that $P_1(\underline m,\underline y),\ldots, P_s(\underline m,\underline y)$ are irreducible in $Q[\underline y]$ and primitive \hbox{w.r.t.} $Z$, and $P_0(\underline m)\not=0$.
This clearly follows from successive applications of the following lemma to each of the variables $t_1,\ldots,t_k$.
\vskip 1mm

\begin{lemma}\label{lem:reduction}
If $Z$ is a near UFD and a Hilbertian ring, then $Z$ has the following property, for all $k,n,s\geq 1$:
\vskip 1mm

\noindent 
${\mathcal HilSch}^{\displaystyle \star}(k,n,s)$: With $\underline t=(t_1,\ldots,t_k)$, $\underline y=(y_1,\ldots,y_n)$, let $P_1(\underline t,\underline y), \ldots, P_s(\underline t,\underline y)$ be $s$ polynomials, irreducible in $Q[\underline t,\underline y]$, of degree $\geq 1$ in $\underline y$ and such that the product $P=P_1\cdots P_s$ has no fixed \underbar {prime} divisor in $Z$ \hbox{w.r.t.} $\underline t$.
Then there is an arithmetic progression $\tau = (\omega \ell + \alpha)_{\ell\in Z}$ ($\omega,\alpha\in Z$, $\omega \not=0$) such that, for infinitely many $m_1 \in \tau$, the polynomials
\vskip 1mm

\noindent
 \hskip 28mm $P_1(m_1,t_2, \ldots, t_{k},\underline y), \ldots, P_s(m_1, t_2, \ldots,t_{k},\underline y)$
\vskip 1mm

\noindent
are irreducible in $Q[t_2, \ldots, t_{k},\underline y]$, of degree $\geq 1$ in $\underline y$, and such that their product, \hbox{i.e.} the polynomial $P(m_1,t_2, \ldots, t_{k},\underline y)$, has no fixed \underbar{prime} divisor in $Z$ \hbox{w.r.t.} $(t_2, \ldots, t_{k})$. 
\end{lemma} 

\noindent
(For the application to Theorem \ref{thm-intro}(a), note that for near UFDs, 
the fixed {prime} divisor condition, both in the assumption and the conclusion of the property, implies that each of the $s$ polynomials in question are primitive \hbox{w.r.t.} $Z$. Also note that
the extra condition in the general definition of Hilbert subsets that $F(\underline m) \not=0$ for some given nonzero polynomial $F\in Q[\underline t]$ is easily guaranteed: in the 
successive applications of Lemma \ref{lem:reduction} to each variable $t_i$, it suffices to exclude finitely many values $m_i$, $i=1,\ldots,k$).

\begin{proof}[Proof of Lemma \ref{lem:reduction}]
Set $\underline t^\prime =(t_2,\ldots,t_k)$. Consider $P_1,\ldots,P_s$ as polynomials in $\underline t^\prime, \underline y$ and with coefficients in $Z[t_1]$. As such, they are primitive \hbox{w.r.t} $Q[t_1]$ (being irreducible in $Q[\underline t,\underline y]$). 
For each $i=1,\ldots,s$, denote by $\delta_i\in Z$ the parameter associated in \S \ref{ssec:preliminaries} with the coefficients of $P_i$ (which are in $Z[t_1]$ and coprime).

Set $\delta =\delta_1\cdots \delta_s$ and let ${\mathcal P}$ be the set of primes $p\in Z$ that divide $\delta$ or such that $|Z/pZ| \leq \max_{2\leq h \leq k} \deg_{t_h}(P)$. 
From Lemma \ref{lem:cond-star},
the set ${\mathcal P}$ is finite (up to units). Let $\omega$ be the product of all primes in  ${\mathcal P}$.

The first step is to construct a $k$-tuple $\underline u= (u_{1},\ldots,u_{k}) \in Z^k$ such that 
\vskip 1mm

\noindent
(1)  \hskip 30mm $P(\underline u,\underline y) \not\equiv 0 \pmod{p}$ for every $p\in {\mathcal P}$.

\vskip 1mm

\noindent 
As $P$ has no fixed {prime} divisor \hbox{w.r.t.} $\underline t$, for every $p\in {\mathcal P}$,
there is a $k$-tuple $\underline u_p \in Z^k$ such that $P(\underline u_p,\underline y)\not\equiv 0 \pmod{p}$.
Using Lemma \ref{lemma:3} and arguing as in the proof of Lemma \ref{lem:reduction1} (1st step), one finds $\underline u_0= (u_{01},\ldots,u_{0k}) \in Z^k$ satisfying (1). Furthermore, denoting by $\tau_h$ the arithmetic progression $\tau_h= (\omega \ell + u_{0h})_{\ell\in Z}$ ($h=1,\ldots,k$), the conclusion holds for every $\underline u = (u_1,\ldots,u_k)\in \tau_1 \times\cdots \times \tau_k$. Fix such a $k$-tuple $\underline u$ and set $\alpha=u_1$.

It follows from the fixed prime divisor assumption \hbox{w.r.t.} $\underline t$ that $P$ has no fixed prime divisor \hbox{w.r.t.} the variable $t_1$.
From Remark \ref{rem:plus2}, we have that, for every $\ell_1\in Z$,

\vskip 1mm

\noindent
(2) \ the polynomials $P_i(\omega \ell_1+\alpha, \underline t^\prime, \underline y)$, $i=1,\ldots,s$, have no prime divisors in $Z$.
\vskip 1mm

\noindent
(Note that condition from Remark \ref{rem:plus2} that $P(\alpha, \underline t^\prime, \underline y) \not\equiv 0 \pmod{p}$
 and  every prime divisor of $\delta$ is guaranteed by (1)).

Consider the following polynomials, where $v_1$ is a new variable:
\vskip 1,5mm

\centerline{$\widetilde P_i (v_1,\underline t^\prime, \underline y)= P_i(\omega v_1+\alpha, \underline t^\prime, \underline y) \in Z[v_1,\underline t^\prime,\underline y],\hskip 2mm  i=1,\ldots,s$.}
\vskip 1mm

\noindent
The polynomials $\widetilde P_1,\ldots, \widetilde P_s$ are irreducible in $Q[\underline t^\prime][v_1,\underline y]$ and of degree at least $1$ in $\underline y$. As $Z$ is a Hilbertian ring, there exist infinitely many $\ell_1\in Z$ such that 

\vskip 1mm

\noindent
(3) the polynomials $\widetilde P_i(\ell_1,\underline t^\prime,\underline y)$ ($i=1,\ldots,s$) 
are irreducible in $Q[\underline t^\prime,\underline y]$, of degree $\geq 1$ in $\underline y$.
\vskip 1,5mm

Fix $\ell_1\in Z$ as in (3) and set $\widetilde P = \prod_{i=1}^s \widetilde P_i$.
To end the proof, it remains to check that $\widetilde P(\ell_1,\underline t^\prime,\underline y)$ has no fixed prime divisor
\hbox{w.r.t.} $\underline t^\prime$.
Assume that for some prime $p\in Z$, we have $\widetilde P(\ell_1, \underline m^\prime, \underline y) \equiv 0 \pmod{p}$ for every $\underline m^\prime \in Z^{k-1}$. Note that due to (1) we have
\vskip 1,5mm

\noindent
\hskip 28mm $\widetilde P(\ell_1, u_2, \ldots, u_k,\underline y) \not\equiv 0 \pmod{p}$, for every $p\in {\mathcal P}$. 
\vskip 1,5mm

\noindent
Therefore $p\notin {\mathcal P}$. Hence, by choice of ${\mathcal P}$, 
\vskip 1mm

\noindent
(4) \hskip 20mm $\displaystyle |Z/pZ| > \max_{2\leq h \leq k} \deg_{t_h}(P)  \geq  \max_{2\leq h \leq k} \deg_{t_h} \left(\widetilde P(\ell_1, \underline t^\prime, \underline y)\right)$
\vskip 1mm

\noindent
By (2), $\widetilde P_1(\ell_1, \underline t^\prime, \underline y), \ldots,\widetilde P_s(\ell_1, \underline t^\prime, \underline y)$ 
are nonzero modulo $p$ (in $(Z/pZ)[\underline t^\prime,\underline y]$). Hence so is their product $\widetilde P(\ell_1, \underline t^\prime, \underline y)$. This, conjoined with (4) contradicts Lemma \ref{lem:prelim-proof-R[u]}(a).
\end{proof}

\subsection{Variants of Theorem \ref{thm-intro}} \label{ssec:refine}

\subsubsection{Relaxing the fixed divisor assumption} \label{get-rid-SC}
Let $Z$ be a near UFD and $P_1,\ldots,P_s$ be as in Definition \ref{def:hilbert-schinzel}, except that ${\mathcal F}_{\underline t}(P_1\cdots P_s) = \emptyset$ is no longer assumed. By Lemma \ref{lem:prelim-proof-R[u]}, the set of primes in ${\mathcal F}_{\underline t}(P_1\cdots P_s)$ is finite (modulo units). Let $\varphi$ be the product of them. One can then conclude that 
\vskip 1,5mm

\noindent
(*)  {\it there is a Zariski-dense subset $H\subset Z[1/\varphi]^k$ such that for every $\underline m \in H$, the polynomials $P_1(\underline m,\underline y),\ldots, P_s(\underline m,\underline y)$ are irreducible in $Q\hskip 1pt [\underline y\hskip 1pt]$ and primitive \hbox{w.r.t.} $Z[1/\varphi]$.}
\vskip 1,5mm

\noindent
Of course, if ${\mathcal F}_{\underline t}(P_1\cdots P_s) = \emptyset$, then $\varphi=1$ and we merely have Theorem \ref{thm-intro}. Conversely, the improved conclusion follows from Theorem \ref{thm-intro} by just taking $Z$ to be $Z[1/\varphi]$: simply note that the assumptions on $Z$ are preserved by passing to  $Z[1/\varphi]$.

\begin{remark} 
One can avoid inverting $\varphi$ and still not assume ${\mathcal F}_{\underline t}(P_1\cdots P_s) = \emptyset$, but then specializing to points $\underline m\in Z^k$ should be replaced by specializing to points $\underline m(\underline y)\in Z[\underline y]^k$: a Zariski-dense subset of $\underline m(\underline y)\in Z[\underline y]^k$ exists such that $P_1(\underline m(\underline y),\underline y), \ldots, P_s(\underline m(\underline y),\underline y)$ are irreducible in $Z[\underline y]$ \cite[Theorem 1.1]{BDN19}.
 \end{remark}
 
 \subsubsection{A variant of Theorem \ref{thm-intro} for Dedekind domains with $k\geq 1$} Compared to Theorem \ref{thm-intro}(b) where $k=1$,
 the following result, for  Dedekind domains, has $k\geq 1$ but $s=1$, \hbox{i.e.} concerns one polynomial with $k\geq 1$ variables $t_i$ to 
 be specialized.

 \begin{theorem} \label{thm:dedekind-HS-k} Let $Z$ be a Dedekind domain and a Hilbertian ring. Let $P(\underline t,\underline y)$ be a po\-ly\-no\-mial, irreducible in $Q[\underline t,\underline y]$, of degree $\geq 1$ in $\underline y$ (with $k,n\geq 1$).  Assume that $P$ has no fixed \underbar{prime} divisor \hbox{w.r.t.} $\underline t$.
There is a Zariski-dense subset $H\subset Z^k$ such that for every $\underline m \in H$,  the polynomial $P(\underline m,\underline y)$ is irreducible in $Q\hskip 1pt [\underline y\hskip 1pt]$ and has no \underbar{prime} divisor in $Z$.
\end{theorem}  
 
The cost of the generalization to $k\geq 1$ is that the primitivity \hbox{w.r.t.} $Z$ of the polynomial $P(\underline m,\underline y)$ in the conclusion 
is replaced by the non-divisibility by any prime $p\in Z$. On the other hand the assumption is weaker: 
$P(\underline t,\underline y)$ is not assumed to be primitive \hbox{w.r.t.} $Z$ and the fixed divisor assumption is restricted to primes.

 \begin{proof}
 As before, one may assume that $P$ is not a monomial in $\underline y$. The statement clearly follows from successive applications of property ${\mathcal HilSch}^{\displaystyle \star}(k,n,1)$ from Lemma \ref{lem:reduction}
  to each of the variables $t_1,\ldots,t_k$. In Lemma \ref{lem:reduction}, this property was proved in the case that $Z$ is a near UFD. We prove it below in the case that $Z$ is a Dedekind domain (and a Hilbertian ring as also 
  assumed in Lemma \ref{lem:reduction}). The strategy is the same.
    
Recall that here $s=1$. Let $P_1(t_1),\ldots,P_N(t_1)\in Z[t_1]$ be the coefficients of $P$ viewed as a polynomial in $\underline t^\prime, \underline y$, where $\underline t^\prime = (t_2,\ldots, t_k)$. They are coprime in $Q[t_1]$ (a consequence of $P$ being irreducible in $Q[\underline t,\underline y]$). Let $\delta\in Z$ be the associated parameter from \S \ref{ssec:preliminaries} and let $\mathcal{Q}_1,\ldots, \mathcal{Q}_r$ be the prime ideals of $Z$ dividing $\delta$. One may assume that the first ones, say $\mathcal{Q}_1,\ldots, \mathcal{Q}_\rho$, are principal, generated by prime elements $q_1,\ldots,q_\rho$ respectively, while the last ones $\mathcal{Q}_{\rho+1},\ldots, \mathcal{Q}_r$ are not principal. 

Let ${\mathcal P}$ be the union of  the set $\{q_1,\ldots,q_\rho\}$ and of the set of primes $p\in Z$  such that $|Z/pZ| \leq \max_{2\leq h \leq k} \deg_{t_h}(P)$. 
From Lemma \ref{lem:cond-star},
the set ${\mathcal P}$ is finite. Let $\omega$ be the product of all primes in  ${\mathcal P}$.

 Let $\mathfrak{I}$ be the ideal of $Z$ generated by all values $P_1(z),\dots, P_N(z)$ with $z\in Z$.
 Denote by $g_1,\ldots,g_r\geq 0$ the respective exponents of $\mathcal{Q}_1,\ldots, \mathcal{Q}_r$ in the prime ideal factorization of  $\mathfrak{I}$.

 As $P$ has no fixed {prime} divisor in $Z$ \hbox{w.r.t.} $\underline t$, for every $p\in {\mathcal P}$,
there is a $k$-tuple $\underline u_p= (u_{p1}, \ldots, u_{pk}) \in Z^k$ such that 

\vskip 0,5mm

\centerline{$P(\underline u_p,\underline y)\not\equiv 0 \pmod{p}$.}
\vskip 0,5mm

\noindent 
Consider next the ideals $\mathcal{Q}_{\rho+1}^{g_{\rho+1}+1}, \ldots, \mathcal{Q}_{r}^{g_{r}+1}$. As none of these ideals contains $\mathfrak{I}$, for each $j=\rho+1,\ldots,r$, there exists $i_j\in \{1,\dots, N\}$ and $u_{j1}\in Z$ such that $P_{i_j}(u_{j1})\not\equiv 0 \pmod{\mathcal{Q}_j^{g_j+1}}$, 
or equivalently,
\vskip 0,5mm

\centerline{$(P_1(u_{j1}),\ldots,P_N(u_{j1})) \not\equiv (0,\ldots,0) \pmod {\mathcal{Q}_j^{g_j+1}}$,}
\vskip 0,5mm

\noindent
or, again equivalently,
\vskip 0,5mm

\centerline{$P(u_{j1}, \underline t^\prime, \underline y) \not\equiv 0 \pmod {\mathcal{Q}_j^{g_j+1}}$.}
\vskip 0,5mm

Any two ideals in the set ${\mathcal P} \cup \{\mathcal{Q}_{\rho+1}^{g_{\rho+1}},\ldots, \mathcal{Q}_r^{g_\rho}\}$ are comaximal (or this set is empty or a sin\-gle\-ton). The Chinese Remainder Theorem gives a tuple $\underline u= (u_{1},\ldots,u_{k}) \in Z^k$ \hbox{such that}
\vskip 1mm

\noindent
(1) (a) \hskip 30mm $P(\underline u,\underline y) \not\equiv 0 \pmod{p}$, \hskip 1mm  for each $p\in {\mathcal P}$,

\hskip 1mm (b) \hskip 20mm and  $P(u_1,\underline t^\prime,\underline y) \not\equiv 0 \pmod{\mathcal{Q}_j^{g_j+1}}$, \hskip 1mm for each  $j=\rho+1,\ldots, r$.

\vskip 1mm

Note that condition (1-a) implies that $P(u_1,\underline t^\prime,\underline y) \not\equiv 0 \pmod{q_j^{g_j+1}}$, \hskip 1mm for each  $j=1,\ldots, \rho$. So condition (1-b) in fact holds for each $j=1,\ldots,r$. Note further that the fixed prime divisor assumption \hbox{w.r.t.} $\underline t$ implies that $P$ has no fixed prime divisor \hbox{w.r.t.} the variable $t_1$. Set $\alpha = u_1$. Proposition \ref{rem:plus}(b) can be applied to conclude that

\vskip 1mm

\noindent
(2) \ the polynomial $P(\omega \ell_1+\alpha, \underline t^\prime, \underline y)$ has no prime divisors $p\in Z$ for every $\ell_1\in Z$.
 \vskip 1mm

The end of the proof of Lemma \ref{lem:reduction} can now be reproduced {\it mutatis mutandi}.
Consider the following polynomial, where $v_1$ is a new variable:
\vskip 1,5mm

\centerline{$\widetilde P (v_1,\underline t^\prime, \underline y)= P(\omega v_1+\alpha, \underline t^\prime, \underline y) \in Z[v_1,\underline t^\prime,\underline y]$.}
\vskip 1mm

\noindent
The polynomial $\widetilde P$ is irreducible in $Q[\underline t^\prime][v_1,\underline y]$ and of degree at least $1$ in $\underline y$. As $Z$ is a Hilbertian ring, there exist infinitely many $\ell_1\in Z$ such that 

\vskip 1mm

\noindent
(3) the polynomial $\widetilde P(\ell_1,\underline t^\prime,\underline y)$ is irreducible in $Q[\underline t^\prime,\underline y]$, of degree $\geq 1$ in $\underline y$.
\vskip 1,5mm

Fix $\ell_1\in Z$ as in (3). It remains to check that $\widetilde P(\ell_1,\underline t^\prime,\underline y)$ has no fixed prime divisor
\hbox{w.r.t.} $\underline t^\prime$.
Assume that for some prime $p\in Z$, we have $\widetilde P(\ell_1, \underline m^\prime, \underline y) \equiv 0 \pmod{p}$ for every $\underline m^\prime \in Z^{k-1}$. Note that due to (1-a), we have
\vskip 1,5mm

\noindent
\hskip 28mm $\widetilde P(\ell_1, u_2, \ldots, u_k,\underline y) \not\equiv 0 \pmod{p}$, for every $p\in {\mathcal P}$. 
\vskip 1,5mm

\noindent
Therefore $p\notin {\mathcal P}$. Hence, by choice of ${\mathcal P}$, 
\vskip 1mm

\noindent
(4) \hskip 20mm $\displaystyle |Z/pZ| > \max_{2\leq h \leq k} \deg_{t_h}(P)  \geq  \max_{2\leq h \leq k} \deg_{t_h} \left(\widetilde P(\ell_1, \underline t^\prime, \underline y)\right)$
\vskip 1mm

\noindent
As, by (2), $\widetilde P(\ell_1, \underline t^\prime, \underline y)$ 
is nonzero modulo $p$, this contradicts Lemma \ref{lem:prelim-proof-R[u]}(a).
 \end{proof}

\bibliographystyle{alpha}
\bibliography{Debes-et-al_11590.bib}
\end{document}